\documentclass[10pt,a4paper]{article}

\usepackage[T1]{fontenc}
\usepackage{amssymb}
\usepackage{amsthm}
\usepackage{mathtools}
\usepackage{enumitem}
\usepackage{aliascnt}
\usepackage{hyperref}
\numberwithin{equation}{section}
\newtheorem{theorem}{Theorem}[section]

\newaliascnt{proposition}{theorem}

\aliascntresetthe{proposition}

\newaliascnt{lemma}{theorem}
\newtheorem{lemma}[lemma]{Lemma}
\aliascntresetthe{lemma}

\newaliascnt{corollary}{theorem}
\newtheorem{corollary}[corollary]{Corollary}
\aliascntresetthe{corollary}

\newaliascnt{conjecture}{theorem}
\newtheorem{conjecture}[conjecture]{Conjecture}
\aliascntresetthe{conjecture}

\theoremstyle{definition}
\newaliascnt{claim}{theorem}
\newtheorem{claim}[claim]{Claim}
\aliascntresetthe{claim}

\usepackage[noabbrev,capitalize]{cleveref}
\crefname{equation}{}{}

\crefname{theorem}{Theorem}{Theorems}
\crefname{proposition}{Proposition}{Propositions}
\crefname{lemma}{Lemma}{Lemmas}
\crefname{corollary}{Corollary}{Corollaries}
\crefname{conjecture}{Conjecture}{Conjectures}
\crefname{claim}{Claim}{Claims}

\newcommand{\F}{\mathcal F}
\newcommand{\G}{\mathcal G}
\newcommand{\Y}{\mathcal Y}
\newcommand{\eps}{\varepsilon}
\newcommand{\ceil}[1]{\left\lceil #1 \right\rceil}

\begin{document}
\title{A solution to the Frankl--Kupavskii conjecture on the Erd\H{o}s--Kleitman matching problem}

\author{
    Cheng Chi\thanks{School of Mathematical Sciences, Shanghai Jiao Tong University, 800 Dongchuan Road, Shanghai 200240, China.
        Email: chengchi@sjtu.edu.cn.
        Supported by National Key R\&D Program of China under grant No. 2022YFA1006400 and National Natural Science Foundation of China No. 12571376.
    }
    \qquad
    Yan Wang\thanks{School of Mathematical Sciences, Shanghai Jiao Tong University, 800 Dongchuan Road, Shanghai 200240, China.
        Email: yan.w@sjtu.edu.cn.
        Supported by National Key R\&D Program of China under grant No. 2022YFA1006400, National Natural Science Foundation of China under grant No. 12571376 and SJTU-Warwick Joint Seed Fund.
    }
}

\date{}
\maketitle
\begin{abstract}
    For integers $n\ge s\ge2$, let $e(n,s)$ be the maximum size of a family
    $\F\subseteq2^{[n]}$ with no $s$ pairwise disjoint members.
    The problem of determining $e(n,s)$ is closely related to its uniform
    counterpart, the Erd\H{o}s Matching Conjecture.
    Frankl and Kupavskii conjectured an exact formula for $e((m+1)s-\ell,s)$ when $1\le \ell\le \lceil s/2\rceil$.
    We prove that for every fixed $m\ge3$ and sufficiently large $s$, the extremal families for $e((m+1)s-\ell,s)$ are
    \[
        P(m,s,\ell;L)\coloneqq\{A\subseteq [n]\colon |A|+|A\cap L|\ge m+1\} \text{ for some $L$ with } |L|=\ell-1
    \]
    when $1\le \ell\le (\frac{m+1}{2m+1}-o(1))s$.
    In particular, this confirms the Frankl--Kupavskii conjecture for every
    fixed $m\ge3$ and all sufficiently large $s$.
    For $m=3$ and sufficiently large $s$, we determine the exact range of $\ell$ for which $P(3,s,\ell;L)$ is extremal, generalizing a theorem of Kupavskii and Sokolov.
\end{abstract}

\section{Introduction}\label{sec:intro}

Let $[n]\coloneqq \{1,2,\ldots,n\}$.
A \emph{family} is a collection of sets.
For a set $X$ and an integer $k$, we write $\binom Xk$ for the family of
all $k$-element subsets of $X$, and we write
$\binom X{\ge k}\coloneqq \bigcup_{i\ge k}\binom Xi$.
Given an integer $n$, we also write $\binom{n}{\ge k}$ for $\sum_{i=k}^{n}\binom ni$.
A \emph{matching} in a family is a collection of pairwise disjoint sets.
An $s$-\emph{matching} is a matching of size $s$.
Given a family $\F$, its \emph{matching number} $\nu(\F)$ is the size of a
largest matching contained in $\F$.

Extremal problems concerning matching numbers in set systems date back to the classical work of Erd\H{o}s in the 1960s.
One of the representative results is the classical Erd\H{o}s--Ko--Rado theorem \cite{ErdosKoRado}.
Erd\H{o}s \cite{ErdosEMC} proposed the following problem: determining the largest possible size of a $k$-uniform family that contains no given number of
pairwise disjoint members.
This problem is now known as the Erd\H{o}s Matching Conjecture.

For integers $N,k,t$ with $N\ge kt$, write
\[
    e_k(N,t)
    \coloneqq
    \max\left\{|\G|:\G\subseteq\binom{[N]}k
    \text{ and } \nu(\G)<t\right\}.
\]
The Erd\H{o}s Matching Conjecture \cite{ErdosEMC} asserts that
\[
    e_k(N,t)
    =
    \max\left\{
    \binom Nk-\binom{N-t+1}k,\,
    \binom{kt-1}k
    \right\}.
\]
Erd\H{o}s also confirmed the conjecture holds for sufficiently large $N$ when $k$ and $t$ are fixed \cite{ErdosEMC}.
The conjecture is known for $k=2$ as the Erd\H{o}s--Gallai theorem
\cite{ErdosGallai}.
For $k=3$, it was proved for sufficiently large $N$ by \L{}uczak and Mieczkowska \cite{LuczakMieczkowska} and then fully solved by Frankl \cite{FranklDAM}.
Despite this substantial progress, the full conjecture remains open for every $k\ge4$.

In this paper, we study the non-uniform analogue of the Erd\H{o}s Matching Conjecture.
For integers $n\ge s\ge2$, define
\[
    e(n,s)
    \coloneqq
    \max\{|\F|:\F\subseteq2^{[n]}\text{ and }\nu(\F)<s\}.
\]
The study of this problem, now called the Erd\H{o}s--Kleitman problem, goes back to a question of Erd\H{o}s answered by Kleitman, who determined $e(sm-1,s)$ and $e(sm,s)$ for all $m,s\ge1$ \cite{Kleitman}.
In contrast with these two residue classes, the other values of $e(n,s)$ remain unknown and may depend heavily on $n\pmod s$.

The following construction explains the connection between the Erd\H{o}s--Kleitman problem and its uniform analogue.
Write $n=(m+1)s-\ell$.
Suppose that $\G\subseteq\binom{[n]}m$ contains no $\ell$ pairwise disjoint
members.
Then the lifted family $\G\cup\binom{[n]}{\ge m+1}$ contains no $s$ pairwise disjoint members.
Indeed, if an $s$-matching in this lifted family uses $q$ sets from $\G$, then its total number of vertices is at least $qm+(s-q)(m+1)=n+\ell-q$.
Since the ground set has size $n$, this forces $q\ge\ell$, which would give an $\ell$-matching in $\G$.
Thus, when $n=(m+1)s-\ell$, an $m$-uniform family with no $\ell$-matching can be lifted to a non-uniform family with no $s$-matching by adding all sets of size at least $m+1$.
However, this lifting only accounts for the $m$-th layer together with all higher layers.
The construction below contains lower-layer sets, which makes the Erd\H{o}s-Kleitman problem more difficult.
For each $L\in\binom{[n]}{\ell-1}$, define
\[
    P(m,s,\ell;L)
    \coloneqq
    \{A\subseteq[n]: |A|+|A\cap L|\ge m+1\}.
\]
The size of this family is independent of the choice of $L$ and we denote this
size by $|P(m,s,\ell)|$.

The family $P(m,s,\ell;L)$ has no $s$ pairwise disjoint members.
Indeed, if $A_1,\ldots,A_s$ were pairwise disjoint members of $P(m,s,\ell;L)$, then $\sum_{i=1}^s |A_i|\le n=(m+1)s-\ell$ and $\sum_{i=1}^s |A_i\cap L|\le |L|=\ell-1$.
Adding the two inequalities gives $\sum_{i=1}^s\bigl(|A_i|+|A_i\cap L|\bigr)\le (m+1)s-1$, contradicting the condition that
$|A_i|+|A_i\cap L|\ge m+1$ for every $i$.

Thus $P(m,s,\ell;L)$ gives a natural lower bound for $e(n,s)$.
Frankl and Kupavskii conjectured that, in the following range, this lower bound is always sharp.

\begin{conjecture}[Frankl and Kupavskii \cite{FKnonuniform}]\label[conjecture]{conj}
    Suppose that $s\ge2$, $m\ge1$, and
    $n=(m+1)s-\ell$ for some integer $\ell$ with
    $1\le \ell\le\ceil{s/2}$.
    Then
    \(e(n,s)=|P(m,s,\ell)|.\)
\end{conjecture}

They \cite{FKnonuniform} also confirmed the conjecture for several $(s,\ell)$.

\begin{theorem}[Frankl and Kupavskii \cite{FKnonuniform}]
    \label[theorem]{thm:FK-small-ell}
    $e((m+1)s-\ell,s)=|P(m,s,\ell)|$ holds in each of the following cases:
    \begin{enumerate}
        \item $\ell=2$ and $s\ge5$, and also $\ell=2$, $s=4$ when $m$ is even;
        \item $m=1$;
        \item $s\ge \ell m+3\ell+3$.
    \end{enumerate}
\end{theorem}

The case $m=2$ was completely solved recently by Kupavskii and Sokolov
\cite{KupavskiiSokolov2025}.
Thus $m=3$ is the first unresolved value in the conjectured range.
Very recently, Kupavskii and Sokolov also determined the extremal families for $e(sm+c,s)$ when $c$ is sufficiently small compared to $s$ \cite{KupavskiiSokolovOtherEnd}.
Moreover, they show that \cref{conj} holds for every $m$, every $\ell$ with $1\le \ell\le (1/2-o(1))s$ and sufficiently large $s$ \cite{KupavskiiSokolovMore}.
Therefore, \cref{conj} holds asymptotically, provided that $s$ is sufficiently large.

Our main contribution is to prove that, for every fixed $m\ge3$ and all
sufficiently large $s$, the family $P(m,s,\ell;L)$ is the
unique extremal family when $1\le \ell\le (\frac{m+1}{2m+1}-o(1))s$.
This range properly contains the conjectured interval
$1\le\ell\le\ceil{s/2}$ and extends beyond it.
Thus we give an affirmative answer to \cref{conj} due to Frankl-Kupavskii in a strong sense.

\begin{theorem}\label{thm:main}
    Fix $m\ge3$ and $\eps>0$.  There is an integer
    $s_0=s_0(m,\eps)$ such that the following holds for all $s\ge s_0$.
    Suppose
    \[
        1\le \ell\le \left(\frac{m+1}{2m+1}-\eps\right)s,
        \qquad
        n=(m+1)s-\ell.
    \]
    If $\F\subseteq2^{[n]}$ satisfies $\nu(\F)<s$, then
    \(|\F|\le |P(m,s,\ell)|.\)
    Moreover, equality holds if and only if $\F=P(m,s,\ell;L)$
    for some $L\in\binom{[n]}{\ell-1}$.
\end{theorem}

The general theorem is not sharp when $m=3$.
In this case, the exact three-uniform matching theorem allows the comparison to continue until a second natural construction becomes competitive.
Let $n=4s-\ell$ for some $\ell$ with $1\le \ell\le s$ and $L'\in\binom{[n]}{3\ell-1}$.  Define
\[
    P'(s,\ell;L')
    =
    \binom{L'}3\cup\binom{[n]}{\ge4}.
\]
Its size is independent of the choice of $L'$ and we denote this size by $|P'(s,\ell)|$.

Since $P(3,s,\ell;L)$ and $P'(s,\ell;L')$ agree on all layers of size at least $4$, it suffices to compare the size of third layers of $P(3,s,\ell;L)$ and $P'(s,\ell;L')$, together with the additional contribution from the second layer of $P(3,s,\ell;L)$.
Let $t(s)$ be the root of $|P(3,s,\ell)|=|P'(s,\ell)|$ with $t(s)>1$.
A direct calculation gives
\[
    t(s)
    \coloneqq
    \frac{17-18s+\sqrt{49-852s+1284s^2}}{20}
    =
    \frac{\sqrt{321}-9}{10}s+O(1)\approx 0.8916s+O(1).
\]
For $m=3$, this leads to the following sharper theorem.
\begin{theorem}\label{thm:m3-lower-range}
    There exists an integer $s_0$ such that the following statements hold for
    all $s,\ell$ with $s\ge s_0$ and $1\le \ell\le t(s)$.
    Write $n\coloneqq 4s-\ell$.
    If $\F\subseteq2^{[n]}$ satisfies $\nu(\F)<s$, then
    \(|\F|\le |P(3,s,\ell)|.\)
    Moreover, when $\ell<t(s)$, equality holds if and only if
    $\F=P(3,s,\ell;L)$ for some $L\in\binom{[n]}{\ell-1}$.
    If $t(s)$ is an integer and $\ell=t(s)$, equality holds if and only if
    either $\F=P(3,s,\ell;L)$ for some
    $L\in\binom{[n]}{\ell-1}$ or
    $\F=P'(s,\ell;L')$ for some $L'\in\binom{[n]}{3\ell-1}$.
\end{theorem}

The two quantities $|P(3,s,\ell)|$ and
$|P'(s,\ell)|$ are equal precisely at the two roots $\ell=1$ and $\ell=t(s)$.
Moreover, $|P(3,s,\ell)|>|P'(s,\ell)|$ for $1<\ell<t(s)$, while
$|P(3,s,\ell)|<|P'(s,\ell)|$ when $s\ge\ell>t(s)$.
Therefore, for $t(s)<\ell\le s$, the construction $P'(s,\ell;L')$ is better
than $P(3,s,\ell;L)$ while still avoiding an $s$-matching, so
$P(3,s,\ell;L)$ is not extremal.
Thus \cref{thm:m3-lower-range} determines the exact range of $\ell$ for which
$P(3,s,\ell;L)$ is extremal.

The rest of the paper is organized as follows.
In \cref{sec:lems}, we collect the uniform matching input and prove the
auxiliary blocker lemma.
In \cref{sec:comparison}, we set up the layer-by-layer comparison with the
canonical family $P(m,s,\ell;L)$.
The proof of \cref{thm:main} is given in \cref{sec:proof-of-main-result-1}.
\cref{thm:m3-lower-range} is proved in \cref{sec:m3-threshold}.
Some numerical comparisons needed in the proof of \cref{thm:m3-lower-range} are presented in \cref{app:m3-numerics}.

\section{Definitions and lemmas}\label{sec:lems}

In this section, we collect some useful notation and several lemmas.
For a family $G$, let $\tau(G)$ denote its vertex-cover number, namely the minimum size of a set of vertices meeting every member of $G$; we use the convention $\tau(\emptyset)=0$.

We first record several theorems concerning the Erd\H{o}s Matching Conjecture.

\begin{lemma}[Erd\H{o}s and Gallai \cite{ErdosGallai}]
    \label[lemma]{lem:EG}
    If $G\subseteq\binom{[n]}{2}$ with $\nu(G)\le s$, then
    \[
        |G|\le \max\left\{\binom{2s+1}{2},\binom{n}{2}-\binom{n-s}{2}\right\}.
    \]
\end{lemma}

\begin{theorem}[Erd\H{o}s \cite{ErdosEMC}]
    \label{thm:EM-fixed}
    Fix integers $k\ge2$ and $t\ge1$.  There exists $N_0=N_0(k,t)$ such that
    the following holds for all $N\ge N_0$.  If
    $\G\subseteq\binom{[N]}k$ satisfies $\nu(\G)<t$, then
    \(|\G|\le \binom Nk-\binom{N-t+1}k.\)
\end{theorem}

\begin{theorem}[Frankl and Kupavskii \cite{FK2022}]
    \label{thm:FK2022-emc}
    Fix an integer $k\ge 3$. Then there exists $s_{\mathrm{FK}}=s_{\mathrm{FK}}(k)$ such that the
    following holds for every integer $s\ge s_{\mathrm{FK}}$. If
    $\G\subseteq\binom{[n]}{k}$ with $\nu(\G)\le s$ and
    \(n\ge \frac{5}{3}ks-\frac{2}{3}s,\)
    then
    \(|\G|\le \binom{n}{k}-\binom{n-s}{k}.\)
\end{theorem}

\begin{theorem}[Frankl \cite{FranklDAM}]
    \label{thm:m3-EMC3}
    Let $N$ and $t$ be positive integers with $N\ge 3t-1$.
    If $\G\subseteq\binom{[N]}3$ satisfies $\nu(\G)<t$, then
    \[
        |\G|
        \le
        \max\left\{
        \binom N3-\binom{N-t+1}3,
        \binom{3t-1}3
        \right\}.
    \]
    Moreover, if
    \(\binom{3t-1}3> \binom N3-\binom{N-t+1}3\)
    and equality holds with $|\G|=\binom{3t-1}3$, then
    \(\G=\binom U3\)
    for some $U\in\binom{[N]}{3t-1}$.
\end{theorem}

The following corollary combines the preceding results in the form needed below.
The point of this corollary is that it allows the matching number $t$ to be small.

\begin{corollary}
    \label[corollary]{cor:large-ground-emc}
    Fix $k\ge2$ and $\rho>0$.  There exists $N_0=N_0(k,\rho)$ such that the
    following holds for all integers $N\ge N_0$ and $t\ge1$.  If
    $\G\subseteq\binom{[N]}k$ satisfies $\nu(\G)<t$ and
    \(N\ge \left(\frac{5k-2}{3}+\rho\right)t,\)
    then
    \(|\G|\le \binom Nk-\binom{N-t+1}k.\)
\end{corollary}

\begin{proof}
    If $t=1$, then $\nu(\G)<1$, so $\G=\emptyset$ and the desired bound is
    immediate.

    If $k=2$, then  \cref{lem:EG} gives $|\G|\le
        \max\{\binom{2t-1}{2},
        \binom N2-\binom{N-t+1}{2}\}$.
    The second term dominates the first whenever
    \(N\ge \frac{5t-2}{2}.\)
    This follows from $N\ge (8/3+\rho)t$.
    Hence the required bound holds for $k=2$.

    Now suppose $k\ge3$.
    Let $s_{\mathrm{FK}}(k)$ be the threshold in
    Theorem~\ref{thm:FK2022-emc}, and put $t_0=s_{\mathrm{FK}}(k)+1$.
    If $t\ge t_0$, then $t-1\ge s_{\mathrm{FK}}(k)$ and
    $\nu(\G)<t$ gives $\nu(\G)\le t-1$.  Moreover, the hypothesis gives
    \[
        N\ge \left(\frac{5k-2}{3}+\rho\right)t
        > \frac{5k-2}{3}(t-1),
    \]
    so all hypotheses of Theorem~\ref{thm:FK2022-emc}, with matching parameter
    $t-1$, are satisfied.  The theorem yields
    \(|\G|\le \binom Nk-\binom{N-t+1}k.\)

    It remains to handle $2\le t<t_0$.  For each such $t$,
    Theorem~\ref{thm:EM-fixed} gives the same conclusion once
    $N\ge N_0(k,t)$.  Taking the maximum of these finitely many thresholds
    gives a single $N_0(k,\rho)$.
\end{proof}

For integers $k\ge3$ and $N\ge u+k$, put
\[
    h_k(N,u) \coloneqq
    \binom Nk-\binom{N-u}k+1-\binom{N-u-k}{k-1}.
\]
Now we record the following Hilton-Milner type results.
They can be seen as stability results concerning the Erd\H{o}s Matching Conjecture.

\begin{theorem}[Guo, Lu and Mao \cite{GuoLuMao}]
    \label{thm:m3-GLM}
    There exists an integer $N_0$ such that the following holds for all
    integers $N\ge N_0$ and $u\ge1$ with $N\ge3u+2$.
    If $G\subseteq\binom{[N]}3$ satisfies
    \(\nu(G)\le u<\tau(G),\)
    then
    \(|G|\le \max\left\{h_3(N,u),\binom{3u+2}{3}\right\}.\)
\end{theorem}

\begin{theorem}[Frankl and Kupavskii \cite{FKstability}]
    \label{thm:FK-HM}
    Suppose that $k\ge 3$.  For every fixed $\eta>0$ there exists an integer
    $s_1=s_1(k,\eta)$ such that the following holds.  Let $n$ and $s$ be
    integers, and let $G\subseteq\binom{[n]}{k}$ satisfy
    $\nu(G)=s<\tau(G)$.  If either
    \[
        n\ge (s+\max\{25,2s+2\})k,
    \]
    or
    \[
        s\ge s_1(k,\eta)
        \qquad\text{and}\qquad
        n\ge (2+\eta)sk,
    \]
    then $|G|\le h_k(n,s)$.
\end{theorem}

\begin{corollary}
    \label[corollary]{cor:VC}
    Fix $k\ge3$ and $\rho>0$.  There exists a constant
    $N_0=N_0(k,\rho)$ such that the following holds for all integers
    $N\ge N_0$ and $t\ge1$.  If
    $G\subseteq\binom{[N]}k$ satisfies $\tau(G)>t\ge \nu(G)$ and $N\ge (2k+\rho)t$, then
    \[
        |G|
        \le
        \binom{N}{k}
        -
        \binom{N-t}{k}
        -
        \binom{N-t-k}{k-1}
        +
        N^{k-2}.
    \]
\end{corollary}

\begin{proof}
    Put $u \coloneqq \nu(G)$.  Since $\tau(G)>t\ge u$, we have $u<\tau(G)$.
    Also $u\ge1$, because otherwise $G=\emptyset$ and $\tau(G)=0$.

    We first verify that Theorem~\ref{thm:FK-HM} applies to $G$ with matching
    parameter $u$.  Let $\eta=\rho/k$, and let $u_1=s_1(k,\eta)$ be the
    threshold in the second part of Theorem~\ref{thm:FK-HM}.  If $u\ge u_1$,
    then
    \[
        N\ge (2k+\rho)t\ge (2k+\rho)u
        =\left(2+\frac{\rho}{k}\right)ku
        =(2+\eta)ku,
    \]
    so the second part of Theorem~\ref{thm:FK-HM} applies.  If
    $1\le u<u_1$, then the first part applies once
    \(N\ge k\bigl(u+\max\{25,2u+2\}\bigr).\)
    We therefore choose $N_0=N_0(k,\rho)\ge1$; when $u_1>1$, we also require
    \[
        N_0\ge
        \max_{1\le v<u_1} k\bigl(v+\max\{25,2v+2\}\bigr).
    \]
    Thus, for all $N\ge N_0$,
    Theorem~\ref{thm:FK-HM} gives
    \[
        |G|
        \le
        h_k(N,u)
        =
        \binom Nk
        -
        \binom{N-u}k
        +
        1
        -
        \binom{N-u-k}{k-1}.
    \]
    Since $u\le t$, we have $\binom{N-u}k\ge \binom{N-t}k$ and $\binom{N-u-k}{k-1}\ge \binom{N-t-k}{k-1}$.
    Therefore
    \[
        |G|
        \le
        \binom Nk
        -
        \binom{N-t}k
        -
        \binom{N-t-k}{k-1}
        +1
        \le
        \binom Nk
        -
        \binom{N-t}k
        -
        \binom{N-t-k}{k-1}
        +N^{k-2}.
    \]
    This proves the claim.
\end{proof}

The next input converts the absence of a perfect matching into a lower bound on missing edges.
It is the mechanism that later pays for extra low-layer sets.
\begin{lemma}[Baranyai \cite{Baranyai}]
    \label[lemma]{lem:baranyai}
    For integers $q,t\ge 1$, the complete $q$-uniform hypergraph on $qt$
    vertices can be decomposed into $\binom{qt-1}{q-1}$ perfect matchings.
\end{lemma}

We use this consequence in the following blocker lemma.
\begin{lemma}
    \label[lemma]{lem:blocker}
    Let $q\ge 2$, let $G\subseteq\binom{X}{q}$, let $|X|=qt+d$ with $t\ge 1$
    and $d\ge 0$, and put $Z \coloneqq \binom{X}{q}\setminus G$.
    If $\nu(G)<t$, then
    \[
        |Z|
        \ge
        \max\left\{
        \frac1t\binom{qt+d}{q},
        \binom{d+q}{q}
        \right\}.
    \]
\end{lemma}

\begin{proof}
    Let $U$ be a uniformly random $qt$-subset of $X$. By
    Lemma~\ref{lem:baranyai}, the family $\binom{U}{q}$ decomposes into $\binom{qt-1}{q-1}$ perfect matchings. Since $\nu(G)<t$, no such perfect matching can lie entirely in $G$.
    Therefore
    \[
        \left|Z\cap\binom{U}{q}\right|
        \ge
        \binom{qt-1}{q-1}.
    \]

    Taking expectations,
    \[
        \mathbb E\left|Z\cap\binom{U}{q}\right|
        =
        |Z|\,
        \frac{\binom{qt+d-q}{qt-q}}{\binom{qt+d}{qt}}
        \ge
        \binom{qt-1}{q-1}.
    \]
    Hence
    \[
        |Z|
        \ge
        \frac{\binom{qt+d}{qt}\binom{qt-1}{q-1}}
        {\binom{qt+d-q}{qt-q}}
        =
        \frac1t\binom{qt+d}{q}.
    \]

    Let $M$ be a maximum matching in $G$. Then $|M|\le t-1$, and by
    maximality every edge of $G$ meets $V(M)$. Therefore every $q$-set
    contained in $X\setminus V(M)$ belongs to $Z$. Since
    \(|X\setminus V(M)| \ge qt+d-q(t-1) = d+q,\)
    we obtain $|Z|\ge \binom{d+q}{q}$.
\end{proof}

\section{Comparison setup}\label{sec:comparison}

Now we record the notation and assumptions used in the proof of
\cref{thm:main}.
The proof of \cref{thm:m3-lower-range} in \cref{sec:m3-threshold} will use the same notation, but with a different
range for $\ell$.
Assume from now on that $\eps$ is given with $0<\eps<(m+1)/(2m+1)$.
Put
\(\theta_m \coloneqq \frac{m+1}{2m+1}\) and \(\gamma \coloneqq \theta_m-\eps\).
Choose $s_0=s_0(m,\eps)$ sufficiently large.
Let $s,\ell$ be integers with $s\ge s_0$ and $1\le \ell \le \gamma s$.
Let $n=(m+1)s-\ell$.
Put
\(a \coloneqq \ell-1\) and \(r \coloneqq n-a=n-\ell+1=(m+1)s-2\ell+1\).

Given a family $\mathcal A\subseteq 2^{[n]}$ and a positive integer $k$, define
\begin{equation*}
    \mathcal A_i\coloneqq  \mathcal A\cap\binom{[n]}i,
    \quad
    \mathcal A_{\ge k}\coloneqq \bigcup_{i=k}^n \mathcal A_i,
    \quad
    \mathcal A_{\le k}\coloneqq \bigcup_{i=0}^k \mathcal A_i.
\end{equation*}
For a set $E\subseteq[n]$, write $\mathcal A(\overline E)
    \coloneqq
    \{Q\in\mathcal A:Q\cap E=\emptyset\}$ for the subfamily of $\mathcal A$ disjoint from $E$.
Let $\F\subseteq2^{[n]}$ be a family with $\nu(\F)<s$.
Let $\Y$ be the complement family of $\F$, that is, $\Y\coloneqq 2^{[n]}\setminus \F$.
Later, we shall frequently use the notation $\F_i,\Y_i,\F_{\le m},\F_{\ge m+1},\Y_{\ge m+1}$.

For a set $E$ with $|E|\le m$, define its \emph{deficit} by
\(\Delta(E) \coloneqq m+1-|E|.\)
Thus the empty set has deficit $m+1$, a singleton has deficit $m$, an
$(m-1)$-set has deficit $2$, and an $m$-set has deficit $1$.
For a family $\mathcal Q\subseteq\binom{[n]}{\le m}$, define its deficit by
\[
    \Delta(\mathcal Q)\coloneqq\sum_{Q\in\mathcal Q}\Delta(Q).
\]

For $0\le j\le m$, the number of $j$-sets in the canonical family $P(m,s,\ell;L)$ is
\[
    \Lambda_j(a,r)
    \coloneqq
    \left|
    \left\{E\in\binom{[n]}j: |E\cap L|\ge m+1-j\right\}
    \right|
    =
    \sum_{u=m+1-j}^{j}\binom au\binom r{j-u},
\]
where impossible binomial coefficients are interpreted as zero.  In particular,
\[
    \Lambda_0(a,r)=0,
    \qquad
    \Lambda_m(a,r)=\sum_{u=1}^{m}\binom au\binom r{m-u}
    =\binom nm-\binom rm.
\]
Put $\Lambda \coloneqq \sum_{j=0}^{m}\Lambda_j(a,r)$.
Thus $\Lambda$ is the total contribution of the canonical family in the layers of size at most $m$.
Because $P(m,s,\ell;L)$ contains every set of size at least $m+1$, the global comparison reduces to a single inequality: any surplus of $\F$ in the first $m$ layers must be paid for by missing sets of $\F$ in the layers of size at least $m+1$.
Consequently,
\[
    |P(m,s,\ell)|=
    \Lambda+
    \sum_{i\ge m+1}\binom ni.
\]

This reduction is captured by the following lemma.

\begin{lemma}\label[lemma]{lem:low-layer-comparison}
    If $|\F_{\le m}|\le \Lambda+|\Y_{\ge m+1}|$,
    then $|\F|\le |P(m,s,\ell)|$.
\end{lemma}

\begin{proof}
    Since
    \(|\F_{\ge m+1}|= \sum_{i\ge m+1}\binom ni-|\Y_{\ge m+1}|,\)
    we have
    \[
        |\F| =|\F_{\le m}|+|\F_{\ge m+1}|  \le
        \Lambda+|\Y_{\ge m+1}|+
        \left(\sum_{i\ge m+1}\binom ni-|\Y_{\ge m+1}|\right) =|P(m,s,\ell)|.
    \]
    This proves the desired result.
\end{proof}

The next lemma records the coefficient comparison used below.
\begin{lemma}\label[lemma]{lem:coefficient-comparison}
    For every $0\le\alpha\le \theta_m-\eps$,
    \[
        \frac{(m+1-2\alpha)^{m-1}}{(m-1)!}
        >
        \frac{\alpha}{2}\frac{(m+1-\alpha)^{m-2}}{(m-2)!}
    \]
    with a positive gap depending only on $m$ and $\eps$.
\end{lemma}

\begin{proof}
    It is enough to prove the inequality at the larger endpoint
    $\alpha=\theta_m$.  Indeed, the left-hand side is strictly decreasing in
    $\alpha$, while the right-hand side is strictly increasing on
    $[0,\theta_m]$.  Therefore the inequality at $\alpha=\theta_m$ implies it
    throughout $[0,\theta_m-\eps]$.
    At this endpoint the desired inequality is equivalent to
    \((2m-1)^{m-1}>(m-1)2^{m-3}m^{m-2}.\)
    Since $2m-1>2(m-1)$, the left-hand side is larger than
    $2^{m-1}(m-1)^{m-1}$.  It is therefore enough to prove
    \(4(m-1)^{m-2}>m^{m-2}.\)
    For $m=3$ this is $8>3$, while for $m\ge4$ it follows from $\left(\frac{m}{m-1}\right)^{m-2}<e<4$.
    The strict inequality at the endpoint
    and compactness of $[0,\theta_m-\eps]$ give the required positive gap.
\end{proof}

The deficit notation measures how many vertices are saved by using a low-layer set instead of an $(m+1)$-set.
If a disjoint collection of low-layer sets saves at least $\ell$ vertices in total, then the remaining vertices have exactly the form required by the blocker lemma, and this forces many missing $(m+1)$-sets.

\begin{lemma}\label[lemma]{lem:deficit-completion}
    Fix an integer \(m\ge3\) and a real number \(0<\lambda<1\).  There are
    constants \(c_0=c_0(m,\lambda)>0\) and \(s_1=s_1(m,\lambda)\) such that
    the following holds for all \(s\ge s_1\).  Let
    \[
        1\le \ell\le\lambda s,
        \qquad
        n=(m+1)s-\ell.
    \]
    Suppose \(\F\subseteq2^{[n]}\) satisfies \(\nu(\F)<s\).  If there is a
    pairwise disjoint family \(\mathcal Q=\{Q_1,\ldots,Q_q\}\subseteq\F_{\le m}\) such that
    \(\Delta(\mathcal Q)\ge \ell,\)
    then \(|\Y_{m+1}|\ge c_0s^m\).
\end{lemma}

\begin{proof}
    Choose an inclusion-minimal subfamily, still denoted by
    \(\mathcal Q=\{Q_1,\ldots,Q_q\}\), with \(\Delta(\mathcal Q)\ge \ell\).  Since the
    chosen sets have size at most \(m\), each deficit lies between \(1\) and
    \(m+1\).  Minimality gives \(\ell\le \Delta(\mathcal Q)\le \ell+m\).  Moreover,
    each \(Q_i\) has deficit at least one.  If \(q>\ell\), then after removing
    any one of the \(Q_i\) the remaining family has deficit at least
    \(q-1\ge\ell\), contradicting minimality.  Hence \(q\le\ell\).  Write
    \(\Delta(\mathcal Q)=\ell+d\), where \(0\le d\le m\).
    This choice ensures that the selected low-layer sets use exactly
    \((m+1)q-\Delta(\mathcal Q)\) vertices, so the remaining vertex set has the cardinality
    required by the blocker lemma.

    The sets \(Q_1,\ldots,Q_q\) use
    \(\sum_{i=1}^q |Q_i|=(m+1)q-\Delta(\mathcal Q)\)
    vertices.  Let \(W\) be the remaining vertex set.  Then
    \(|W|=n-\bigl((m+1)q-\Delta(\mathcal Q)\bigr) =(m+1)(s-q)+d.\)
    Put \(t\coloneqq s-q\).  Since \(q\le\ell\le\lambda s\),
    \(t=s-q\ge s-\ell\ge(1-\lambda)s.\)

    If \(\F_{m+1}[W]\) contained \(t\) pairwise disjoint \((m+1)\)-sets,
    these sets together with \(Q_1,\ldots,Q_q\) would form an \(s\)-matching in
    \(\F\).  Hence \(\nu(\F_{m+1}[W])<t\).  Since \(|W|=(m+1)t+d\),
    \cref{lem:blocker} gives
    \[
        |\Y_{m+1}|
        \ge
        \frac1t\binom{(m+1)t+d}{m+1}
        \ge t^m
        \ge (1-\lambda)^m s^m.
    \]
    The lemma holds with \(c_0=(1-\lambda)^m\).
\end{proof}

In the proof of \cref{thm:main}, we apply
\cref{lem:deficit-completion} with \(\lambda=\gamma\).  In the proof of
\cref{thm:m3-lower-range}, we apply it with a fixed \(\lambda\) satisfying
\(\alpha_*<\lambda<1\), where \(\alpha_*=(\sqrt{321}-9)/10\approx 0.89\).

\section{Proof of Theorem~\ref{thm:main}}\label{sec:proof-of-main-result-1}

We use the notation from \cref{sec:comparison}.
We first reduce to the case \(\emptyset\notin\F\), and hence
\(|\F_0|=0\).  Suppose that \(\emptyset\in\F\).
If all $1$-sets belonged to \(\F\), then \(\F\) would contain an $s$-matching, a contradiction.
Hence there exists \(x\in[n]\) with \(\{x\}\notin\F\).  Replace \(\emptyset\) by \(\{x\}\), and put
\(\F'\coloneqq(\F\setminus\{\emptyset\})\cup\{\{x\}\}.\)
Then \(|\F'|=|\F|\).  Moreover, \(\nu(\F')<s\): if an \(s\)-matching in
\(\F'\) does not use \(\{x\}\), then it is already an \(s\)-matching in
\(\F\); if it uses \(\{x\}\), then replacing \(\{x\}\) by \(\emptyset\) gives
an \(s\)-matching in \(\F\).  Thus it suffices to prove the theorem for
families with \(\emptyset\notin\F\).

Put \(M \coloneqq \Y_m=\binom{[n]}m\setminus \F_m.\)
Thus $M$ is the family of missing $m$-sets.
By Lemma~\ref{lem:low-layer-comparison}, it remains to prove
\begin{equation}\label{eq:case-target}
    |\F_{\le m}|\le \Lambda+|\Y_{\ge m+1}|.
\end{equation}

We prove \eqref{eq:case-target} by analyzing the \(m\)-uniform layer
\(\F_m\).
We split according to the matching number and the vertex-cover number of $\F_m$:
\[
    \begin{array}{ll}
        \text{Case I:}   & \nu(\F_m)\le a\text{ and }\tau(\F_m)\le a, \\[1mm]
        \text{Case II:}  & \nu(\F_m)\le a<\tau(\F_m),                 \\[1mm]
        \text{Case III:} & \nu(\F_m)\ge a+1=\ell.
    \end{array}
\]

\subsection{Proof of Case I}
In this case, $\nu(\F_m)\le a$ and $\tau(\F_m)\le a$.
Let $A$ be a vertex cover of $\F_m$ with $|A|=a$.
Put $R=[n]\setminus A$ and let $|R|=r$.
Since $A$ meets every $m$-set in $\F_m$, we have $\binom Rm\subseteq M$.
This is exactly the missing $m$-layer of the canonical copy
$P(m,s,\ell;A)$.
For each $0\le j\le m-1$, we say that a $j$-set $E$ is a \emph{canonical
    $j$-set} if
\(|E\cap A|\ge m+1-j.\)
Equivalently, $E$ belongs to the $j$-th layer of $P(m,s,\ell;A)$.

For $E\subseteq[n]$, recall that $M(\overline E)=\{Q\in M:Q\cap E=\emptyset\}$ is the subfamily of missing $m$-sets disjoint from $E$.
For $0\le j\le m-1$, define
\begin{equation}\label{eq:Bj-def-new}
    B_j(M) \coloneqq
    \left\{E\in\binom{[n]}j:
    |M(\overline E)|\ge\binom{r+m+1-2j}{m}
    \right\}.
\end{equation}
By definition, a \(j\)-set lies in \(B_j(M)\) exactly when the number of missing \(m\)-sets disjoint from it is at least the threshold in
\eqref{eq:Bj-def-new}.
If $E\in\F_j\cap B_j(M)$, we shall say that $E$ is \emph{bad}.
If
$E\in\F_j\setminus B_j(M)$, we shall say that $E$ is \emph{good}.

We first count the number of sets in $B_j(M)$.

\begin{claim}\label[claim]{cl:counting-from-missing}
    We have
    \begin{equation}\label{eq:counting-from-missing}
        \sum_{j=0}^{m-1}|B_j(M)|
        \le
        \sum_{j=0}^{m-1}\Lambda_j(a,r)+|M|-\binom rm.
    \end{equation}
    Moreover, if equality holds, then $M=\binom Rm$ and for $0\le j\le m-1$,
    \[
        B_j(M)=\left\{E\in\binom{[n]}j: |E\cap A|\ge m+1-j\right\}.
    \]
\end{claim}
\begin{proof}
    Put \(M^+\coloneqq M\setminus\binom Rm\) and
    \(\xi\coloneqq |M^+|=|M|-\binom rm\).
    For \(E\subseteq[n]\), write
    \(M^+(\overline E)=\{Q\in M^+:Q\cap E=\emptyset\}\).

    We first separate the canonical low sets from the non-canonical ones.  Let \(E\in\binom{[n]}j\) for some $j$ with \(0\le j\le m-1\), and write
    \(u=|E\cap A|\), \(v=|E\cap R|\), so \(j=u+v\).
    If \(E\) is canonical, that is, if \(u\ge m+1-j\), then
    \(r-v=r-j+u\ge m+1+r-2j\).
    Since every \(m\)-set contained in \(R\setminus E\) is missing from \(\F_m\), we get
    \[
        |M(\overline E)|\ge\binom{r-v}{m}
        \ge \binom{m+1+r-2j}{m}.
    \]
    Thus every canonical \(j\)-set belongs to \(B_j(M)\), and the number of
    such sets is \(\Lambda_j(a,r)\).

    Let \(\mathcal X\) be the family of non-canonical low sets which lie in
    one of the \(B_j(M)\).  Explicitly,
    \[
        \mathcal X=
        \bigcup_{j=0}^{m-1}
        \left\{
        E\in B_j(M): |E\cap A|<m+1-j
        \right\}.
    \]
    Therefore
    \(\sum_{j=0}^{m-1}|B_j(M)| = \sum_{j=0}^{m-1}\Lambda_j(a,r)+|\mathcal X|.\)
    It remains to prove \(|\mathcal X|\le\xi\).

    For \(E\in\mathcal X\cap\binom{[n]}j\), define its shortage by
    \(h(E)=m+1-j-|E\cap A|\).  Then \(1\le h(E)\le m+1\), and
    \(r+m+1-2j=r-v+h(E)\).  Since \(E\in B_j(M)\), the extra missing
    \(m\)-sets disjoint from \(E\) satisfy
    \[
        |M^+(\overline E)|
        \ge
        \binom{r-v+h(E)}m-\binom{r-v}m.
    \]
    The shortage $h$ determines how many extra missing $m$-sets are needed
    for $E$ to satisfy the defining inequality for $B_j(M)$.  The minimum possible number
    of such extra missing sets is the following quantity.
    For \(1\le h\le m+1\), define
    \begin{equation}\label{eq:case1-dh}
        D_h(r)=\binom{r-m+1+h}{m}-\binom{r-m+1}{m}=h \frac{r^{m-1}}{(m-1)!}+O_m(r^{m-2}).
    \end{equation}
    The function \(x\mapsto\binom{x+h}{m}-\binom xm\) is increasing for \(x\ge m\).
    Hence every \(E\in\mathcal X\) with shortage \(h\) satisfies
    \(\xi=|M^+|\ge |M^+(\overline E)|\ge D_h(r).\)

    We next compare this required number with the total number of possible
    non-canonical sets of shortage at most $h$.  Let \(N_h\) be the number of
    all non-canonical low sets with shortage at most \(h\).  Explicitly,
    \[
        N_h=
        \sum_{j=0}^{m-1}
        \left|
        \left\{
        E\in\binom{[n]}j:
        1\le m+1-j-|E\cap A|\le h
        \right\}
        \right|.
    \]

    For \(h=1\), the condition \(m+1-j-u=1\) gives \(u=m-j\), and hence
    \(j-u=2j-m\).
    Thus the contribution from $j$-layer is $\binom a{m-j}\binom r{2j-m}$.
    This contribution has degree at most \(j\) in \(s\).
    Therefore the only degree \(m-1\) contribution occurs when \(j=m-1\), where \(u=1\) and \(j-u=m-2\), giving $a\binom r{m-2}$.
    All other layers contribute $o(s^{m-1})$.

    For \(2\le h\le m+1\), the degree \(m-1\) contribution again comes only from \(j=m-1\).
    In that layer, $h(E)=2-|E\cap A|$.
    Thus shortage at most \(h\) implies \(|E\cap A|=0\) or \(|E\cap A|=1\), contributing $\binom r{m-1}+a\binom r{m-2}$ to $N_h$ in total.
    The remaining layers contribute $o(s^{m-1})$.

    Hence, we obtain that
    \[
        N_1=\frac{a r^{m-2}}{(m-2)!}+o(s^{m-1}),
        \text{ and }
        N_h=
        \frac{r^{m-1}}{(m-1)!}
        +
        \frac{a r^{m-2}}{(m-2)!}
        +
        o(s^{m-1})
        \text{ for $2\le h\le m+1$.}
    \]
    By \eqref{eq:case1-dh}, since \(a\le\gamma s\), \(r\ge(m+1-2\gamma)s+1\), and
    \[
        \frac{(m-1)a}{r}
        \le
        \frac{(m-1)\gamma}{m+1-2\gamma}+o(1)
        <1,
    \]
    we have, for $1\le h\le m+1$,
    \[
        D_h(r)-N_h
        \ge
        \left(1-\frac{(m-1)\gamma}{m+1-2\gamma}-o(1)\right)
        \frac{r^{m-1}}{(m-1)!}.
    \]
    By the choice of \(s_0\), for $1\le h\le m+1$, we have
    \(N_h\le D_h(r)-1.\)

    The numbers \(D_1(r),\ldots,D_{m+1}(r)\) are strictly increasing, since
    \(D_{h+1}(r)-D_h(r)=\binom{r-m+h+1}{m-1}>0\).  We compare \(\xi\) with
    this sequence.
    If \(\xi<D_1(r)\), then no non-canonical set can belong to any \(B_j(M)\), so \(|\mathcal X|=0\).
    If \(D_h(r)\le\xi<D_{h+1}(r)\) for some \(1\le h<m+1\), then every set in \(\mathcal X\) has shortage at most \(h\), so
    \(|\mathcal X|\le N_h\le D_h(r)-1\le\xi.\)
    Finally, if \(\xi\ge D_{m+1}(r)\), then \(|\mathcal X|\le N_{m+1}\le D_{m+1}(r)-1\le\xi\).
    Thus \(|\mathcal X|\le\xi\), and hence
    \[
        \sum_{j=0}^{m-1}|B_j(M)|
        \le
        \sum_{j=0}^{m-1}\Lambda_j(a,r)+|M|-\binom rm.
    \]
    This proves \eqref{eq:counting-from-missing}.

    It remains to discuss equality.
    If \(\xi>0\), the preceding argument gives \(|\mathcal X|<\xi\).
    Hence equality in \eqref{eq:counting-from-missing} is impossible unless \(\xi=0\), that is, unless \(M=\binom Rm\).

    Assume now that \(M=\binom Rm\).
    For \(E\in\binom{[n]}j\), again write \(u=|E\cap A|\) and \(v=|E\cap R|\).
    Then \(|M(\overline E)|=\binom{r-v}{m}\).
    Hence \(E\in B_j(M)\) if and only if \(\binom{r-v}{m}\ge\binom{r+m+1-2j}{m}\).
    For \(s_0\) sufficiently large, this is equivalent to \(r-v\ge r+m+1-2j\), which is equivalent to \(u\ge m+1-j\).
    Therefore
    \[
        B_j(M)=
        \left\{
        E\in\binom{[n]}j: |E\cap A|\ge m+1-j
        \right\},
    \]
    as claimed.
\end{proof}

We first consider the case in which there are no good sets.
Then \(\F_j\subseteq B_j(M)\) for every \(0\le j\le m-1\).  By
Claim~\ref{cl:counting-from-missing},
\[
    \sum_{j=0}^{m-1}|\F_j|
    \le
    \sum_{j=0}^{m-1}\Lambda_j(a,r)+|M|-\binom rm.
\]
Adding $|\F_m|=\binom nm-|M|$ gives
\[
    |\F_{\le m}|
    \le
    \sum_{j=0}^{m-1}\Lambda_j(a,r)+\binom nm-\binom rm
    =
    \Lambda.
\]
Thus \eqref{eq:case-target} holds in this subcase.

It remains to consider the case in which a good set exists.

\begin{claim}\label[claim]{cl:good-forces-high-missing}
    If there is a good set $E\in\F_j$ for some $0\le j\le m-1$, then $|\Y_{m+1}|\ge c_0s^m$ for some $c_0=c_0(m,\varepsilon)$.
\end{claim}
\begin{proof}
    Fix a good set \(E\in\F_j\), where \(0\le j\le m-1\), and put
    \(p\coloneqq \ell+j-m-1.\)
    Since \(\Delta(E)=m+1-j\), we have $p+\Delta(E)=\ell$.
    Thus $p$ is exactly the number of disjoint $m$-sets needed, together with
    $E$, to reach family deficit $\ell$.

    If \(p\le0\), then the deficit of $E$ is at least \(\ell\).
    Hence the one-element family \(\{E\}\) satisfies the hypothesis of
    Lemma~\ref{lem:deficit-completion}, and the desired bound follows.

    Suppose now that \(p>0\).
    Since \(E\) is good, we have
    \(|M(\overline E)|<\binom{r+m+1-2j}{m}.\)
    Using \(r=n-\ell+1\) and \(p=\ell+j-m-1\), we have
    \(r+m+1-2j=(n-j)-p+1.\)
    Therefore the family
    \[
        G\coloneqq H(\overline E)
        =
        \{Q\in H:Q\cap E=\emptyset\}
        \subseteq\binom{[n]\setminus E}{m}
    \]
    satisfies
    \(|G|> \binom{n-j}{m}-\binom{(n-j)-p+1}{m}.\)

    Since $\gamma<\theta_m$ and
    \[
        \frac{m+1}{\theta_m}-1=2m>\frac{5m-2}{3},
    \]
    choose \(\rho_1=\rho_1(m,\varepsilon)>0\) such that
    \(\frac{5m-2}{3}+\rho_1<\frac{m+1}{\gamma}-1.\)
    Then
    \(m+1-\left(\frac{5m-2}{3}+\rho_1+1\right)\gamma>0.\)
    Since \(p\le\ell\), \(j\le m-1\), and \(n=(m+1)s-\ell\), we get
    \[
        \begin{aligned}
            n-j-\left(\frac{5m-2}{3}+\rho_1\right)p
             & \ge
            (m+1)s-\ell-(m-1)
            -\left(\frac{5m-2}{3}+\rho_1\right)\ell \\
             & \ge
            \left(m+1-\left(\frac{5m-2}{3}+\rho_1+1\right)\gamma\right)s
            -(m-1).
        \end{aligned}
    \]
    The coefficient of \(s\) is positive.  Hence, by the choice of \(s_0\), we have
    \[
        n-j\ge \left(\frac{5m-2}{3}+\rho_1\right)p
        \qquad\text{and}\qquad
        n-j\ge N_0(m,\rho_1),
    \]
    where \(N_0(m,\rho_1)\) comes from in Corollary~\ref{cor:large-ground-emc}.
    Hence $G$ contains $p$ pairwise disjoint $m$-sets.
    Together with \(E\), these sets form a pairwise disjoint family
    \(\mathcal Q\subseteq \F_{\le m}\)
    whose deficit is $\Delta(\mathcal Q)=\Delta(E)+p=\ell$.
    Lemma~\ref{lem:deficit-completion} gives $|\Y_{m+1}|\ge c_0s^m$, as required.
\end{proof}

We now finish the proof of this subcase.
By Claim~\ref{cl:good-forces-high-missing} $|\Y_{m+1}|\ge c_0s^m$ for some \(c_0=c_0(m,\varepsilon)>0\).
Since \(n\le(m+1)s\), there is a constant \(C_1=C_1(m)\) such that
\(\sum_{j=0}^{m-1}\binom nj\le C_1s^{m-1}.\)
By the choice of \(s_0\),
\[
    \sum_{j=0}^{m-1}|\F_j|
    \le
    \sum_{j=0}^{m-1}\binom nj
    <
    c_0s^m
    \le
    |\Y_{m+1}|.
\]
Moreover, since \(A\) is a vertex cover of \(\F_m\), every \(m\)-set contained
in \(R\) is missing from \(\F_m\).  Hence
\(|\F_m|\le \binom nm-\binom rm=\Lambda_m(a,r).\)
Therefore
\[
    |\F_{\le m}|
    =
    |\F_m|+\sum_{j=0}^{m-1}|\F_j|
    <
    \Lambda_m(a,r)+|\Y_{m+1}|
    \le
    \Lambda+|\Y_{\ge m+1}|.
\]
This completes Case I.

\subsection{Proof of Case II}

In this case, \(\nu(\F_m)\le a<\tau(\F_m)\).
Here the $m$-layer still has small matching number, but it is not controlled by any $a$-vertex cover.
We therefore apply \cref{cor:VC}, which gives a loss $L_m$.
We next show that the lower layers cannot compensate this loss unless they satisfy the hypothesis of \cref{lem:deficit-completion}.
If \(a=0\), then \(\nu(\F_m)=0\), so \(\F_m=\emptyset\) and hence \(\tau(\F_m)=0\), a contradiction.
Thus \(a\ge1\).

Recall that $a=\ell-1, n=(m+1)s-\ell$ and $r=n-a=(m+1)s-2\ell+1$ and put \(\alpha\coloneqq\ell/s\).
Since \(\alpha\le\gamma<\theta_m=(m+1)/(2m+1)\), we may choose
\(\rho_2=\rho_2(m,\varepsilon)>0\) such that
\(2m+\rho_2<\frac{m+1}{\gamma}-1.\)
As \(a=\ell-1<\ell\), we have
\[
    \frac na
    >
    \frac{(m+1)s-\ell}{\ell}
    =
    \frac{m+1}{\alpha}-1
    \ge
    \frac{m+1}{\gamma}-1
    >
    2m+\rho_2.
\]
Hence
\begin{equation}\label{eq:caseII-ratio-for-VC}
    n\ge(2m+\rho_2)a.
\end{equation}

Let \(N_0=N_0(m,\rho_2)\) be the threshold in Corollary~\ref{cor:VC}.  By the
choice of \(s_0\), we have \(n\ge N_0\).  Since \(m\ge3\), \(a\ge1\), and
\(\tau(\F_m)>a\ge\nu(\F_m)\), Corollary~\ref{cor:VC} gives
\begin{equation}\label{eq:caseII-Fm-clean}
    |\F_m|
    \le
    \binom nm-\binom rm-L_m,
\end{equation}
where
\(L_m\coloneqq \binom{r-m}{m-1}-n^{m-2}.\)
Since \(r=(m+1-2\alpha)s+1\), we have
\begin{equation}\label{eq:caseII-Lm-main-clean}
    L_m
    \ge
    \frac{(m+1-2\alpha)^{m-1}}{(m-1)!}s^{m-1}
    -O_{m,\varepsilon}(s^{m-2}).
\end{equation}
In particular, \(L_m>0\) by the choice of \(s_0\).

There are two possibilities.  Either the low layers themselves contain enough
disjoint deficit to force many missing high-layer sets, or they do not.  The
first possibility is dominated by $|\Y_{m+1}|$; the second forces
the $(m-1)$-layer to have small matching number.

We first consider the case in which there is a pairwise disjoint family
\(\mathcal Q\subseteq \F_0\cup\F_1\cup\cdots\cup\F_{m-1}\)
such that $\Delta(\mathcal Q)\ge\ell$.
By Lemma~\ref{lem:deficit-completion}, $|\Y_{m+1}|\ge c_0s^m$.
On the other hand, since \(n\le(m+1)s\),
\(\sum_{j=0}^{m-1}\binom nj=O_m(s^{m-1}).\)
Thus, by the choice of \(s_0\),
\[
    \sum_{j=0}^{m-1}|\F_j|
    \le
    \sum_{j=0}^{m-1}\binom nj
    <
    c_0s^m
    \le
    |\Y_{m+1}|.
\]
Using \eqref{eq:caseII-Fm-clean} and \(L_m>0\), we obtain
\(|\F_m| < \binom nm-\binom rm = \Lambda_m(a,r).\)
Therefore
\[
    |\F_{\le m}|
    =
    |\F_m|+\sum_{j=0}^{m-1}|\F_j|
    <
    \Lambda_m(a,r)+|\Y_{m+1}|
    \le
    \Lambda+|\Y_{\ge m+1}|.
\]
Thus \eqref{eq:case-target} holds strictly in this subcase.

Hence, assume that there is no pairwise disjoint family
\(\mathcal Q\subseteq \F_0\cup\F_1\cup\cdots\cup\F_{m-1}\)
such that $\Delta(\mathcal Q)\ge \ell$.
We now bound the lower layers.  Put $T\coloneqq \left\lfloor a/2\right\rfloor$.
Since every \((m-1)\)-set has deficit \(2\), the family \(\F_{m-1}\) cannot
contain \(T+1\) pairwise disjoint members.
Hence $\nu(\F_{m-1})\le T$.

If \(a=1\), then \(T=0\), so
\(\nu(\F_{m-1})=0\) and hence \(\F_{m-1}=\emptyset\).  Suppose \(a\ge2\).
Then \(T+1\le a\), and by \eqref{eq:caseII-ratio-for-VC},
\[
    \frac n{T+1}
    \ge
    \frac na
    \ge
    2m+\rho_2
    \ge
    \frac{5(m-1)-2}{3}+\rho_2.
\]
By the choice of \(s_0\), we also have \(n\ge N_0(m-1,\rho_2)\).  Therefore
Corollary~\ref{cor:large-ground-emc} applies to \(\F_{m-1}\) with
\(k=m-1\), \(N=n\), and \(t=T+1\).  Since
\(\nu(\F_{m-1})<T+1\), it gives
\(|\F_{m-1}| \le \binom n{m-1}-\binom{n-T}{m-1}.\)
This bound is also valid when \(a=1\), because then both sides are zero.

Using
\[
    \binom n{m-1}-\binom{n-T}{m-1}
    =
    \sum_{i=0}^{T-1}\binom{n-1-i}{m-2}
    \le
    T\binom n{m-2},
\]
together with
\(T\le \frac a2<\frac{\ell}{2}=\frac{\alpha s}{2}\) and \(n=(m+1-\alpha)s\),
we obtain
\begin{equation}\label{eq:caseII-Fm1-main-clean}
    |\F_{m-1}|
    \le
    \frac{\alpha}{2}
    \frac{(m+1-\alpha)^{m-2}}{(m-2)!}
    s^{m-1}.
\end{equation}

It remains to compare the loss $L_m$ in the $m$-layer with the largest possible contribution from the $(m-1)$-layer.  After
normalizing by $s^{m-1}$, this is exactly the following coefficient
comparison.
Define
\[
    \Phi(\alpha)\coloneqq
    \frac{(m+1-2\alpha)^{m-1}}{(m-1)!}
    -
    \frac{\alpha}{2}
    \frac{(m+1-\alpha)^{m-2}}{(m-2)!}.
\]
Combining \eqref{eq:caseII-Lm-main-clean} and
\eqref{eq:caseII-Fm1-main-clean}, we get
\[
    L_m-|\F_{m-1}|
    \ge
    \Phi(\alpha)s^{m-1}
    -
    O_{m,\varepsilon}(s^{m-2}).
\]
By Lemma~\ref{lem:coefficient-comparison}, there is a constant
\(c_1=c_1(m,\varepsilon)>0\) such that $\Phi(\alpha)\ge c_1$ for all $\alpha\le\gamma$.
Thus, by the choice of \(s_0\), we obtain
\(L_m-|\F_{m-1}| \ge \frac{c_1}{2}s^{m-1}.\)
By the choice of \(s_0\) again,
\(\sum_{j=0}^{m-2}\binom nj < \frac{c_1}{2}s^{m-1}.\)
Therefore
\(L_m> |\F_{m-1}|+\sum_{j=0}^{m-2}\binom nj \ge \sum_{j=0}^{m-1}|\F_j|.\)
Finally, using \eqref{eq:caseII-Fm-clean},
\[
    |\F_{\le m}|
    =
    |\F_m|+\sum_{j=0}^{m-1}|\F_j|  \le
    \binom nm-\binom rm-L_m+\sum_{j=0}^{m-1}|\F_j|<
    \binom nm-\binom rm  =
    \Lambda_m(a,r)
    \le
    \Lambda
    \le
    \Lambda+|\Y_{\ge m+1}|.
\]
Thus \eqref{eq:case-target} holds strictly in Case II.

\subsection{Proof of Case III}

In this case, \(\nu(\F_m)\ge \ell\).
Write $q\coloneqq \nu(\F_m)=\ell+d$ for some \(d\ge0\).
Since \(\nu(\F)<s\), we have \(q\le s-1\).  Put
\(t\coloneqq s-q=s-\ell-d.\)
Choose \(q\) pairwise disjoint \(m\)-sets $Q_1,\ldots,Q_q\in\F_m$.
Let
\(W=[n]\setminus\bigcup_{i=1}^q Q_i.\)
Then
\(|W| = n-mq = (m+1)s-\ell-m(\ell+d) = (m+1)t+d.\)
If \(\F_{m+1}[W]\) contained \(t\) pairwise disjoint \((m+1)\)-sets, then
those sets together with \(Q_1,\ldots,Q_q\) would form an \(s\)-matching in
\(\F\).  Hence
\(\nu(\F_{m+1}[W])<t.\)
Since \(q\le s-1\), we have \(t\ge1\).  Applying
Lemma~\ref{lem:blocker} to
\(G=\F_{m+1}[W]\subseteq\binom W{m+1}\)
gives
\begin{equation}\label{eq:caseIII-blocker-new}
    |\Y_{m+1}|
    \ge
    \max\left\{
    \frac1t\binom{(m+1)t+d}{m+1},
    \binom{d+m+1}{m+1}
    \right\}.
\end{equation}

Let
\(b\coloneqq s-\ell.\)
Then \(t+d=b\).  Since \(\ell\le\gamma s\), we have
\(b\ge(1-\gamma)s.\)
Choose \(C=C(m)\) such that $\sum_{i=0}^{m}\binom ni\le Cs^m$, which is possible because \(n\le(m+1)s\).  Choose
\(K=K(m,\varepsilon)\) sufficiently large so that
\begin{equation}\label{eq:caseIII-K-choice}
    \frac{K^{m+1}(1-\gamma)^m}{(m+1)!}>C.
\end{equation}
We now split according to the excess \(d=q-\ell\).

\medskip
\noindent\textbf{Subcase 1: \(d\le Kb^{m/(m+1)}\).}
\medskip

Since \(b\le s\), this gives
\(d=O_{m,\varepsilon}\bigl(s^{m/(m+1)}\bigr)=o(s).\)
Also
\(q+1=\ell+d+1\le \gamma s+o(s)\) and \(n=(m+1)s-\ell\ge(m+1-\gamma)s\).
Since
\(\frac{m+1-\gamma}{\gamma}>\frac{5m-2}{3},\)
we may choose \(\rho_4=\rho_4(m,\varepsilon)>0\) such that, by the choice of
\(s_0\),
\(n\ge \left(\frac{5m-2}{3}+\rho_4\right)(q+1)\) and \(n\ge N_0(m,\rho_4)\),
where \(N_0(m,\rho_4)\) is the threshold in Corollary~\ref{cor:large-ground-emc}.
Since \(\nu(\F_m)=q\), Corollary~\ref{cor:large-ground-emc} gives
\[
    |\F_m|
    \le
    \binom nm-\binom{n-q}{m}
    =
    \binom nm-\binom{r-1-d}{m}.
\]
Since $\Lambda_m(a,r)=\binom{n}{m}-\binom{r}{m}$, we have
\[
    \begin{aligned}
        |\F_m|-\Lambda_m(a,r)
         & \le
        \binom rm-\binom{r-1-d}{m} \\
         & =
        \sum_{i=0}^{d}\binom{r-1-i}{m-1}
        \le
        (d+1)\binom r{m-1}
        =
        O_{m,\varepsilon}\bigl((d+1)s^{m-1}\bigr).
    \end{aligned}
\]
The layers below \(m\) contribute at most
\(\sum_{j=0}^{m-1}|\F_j| \le \sum_{j=0}^{m-1}\binom nj = O_m(s^{m-1}).\)
Therefore
\begin{equation}\label{eq:caseIII-small-d-excess}
    |\F_{\le m}|-\Lambda
    \le
    O_{m,\varepsilon}\bigl((d+2)s^{m-1}\bigr)
    =
    O_{m,\varepsilon}\bigl(s^{m-1+m/(m+1)}\bigr)
    =
    o(s^m).
\end{equation}

On the other hand, since \(q\ge\ell\), any \(\ell\) of the sets
\(Q_1,\ldots,Q_q\) form a pairwise disjoint subfamily \(\mathcal Q\subseteq\F_m\) with
\(\Delta(\mathcal Q)=\ell\).  Lemma~\ref{lem:deficit-completion} gives
\(|\Y_{m+1}|\ge c_0s^m\)
for some \(c_0=c_0(m,\varepsilon)>0\).  Combining this with
\eqref{eq:caseIII-small-d-excess}, and enlarging \(s_0\) if necessary, we get
\[
    |\F_{\le m}|
    <
    \Lambda+|\Y_{m+1}|
    \le
    \Lambda+|\Y_{\ge m+1}|.
\]
Thus \eqref{eq:case-target} holds strictly in this subcase.

\medskip
\noindent\textbf{Subcase 2: \(d>Kb^{m/(m+1)}\).}
\medskip

In this case, \eqref{eq:caseIII-blocker-new} gives
\[
    |\Y_{m+1}|
    \ge
    \binom{d+m+1}{m+1}
    \ge
    \frac{d^{m+1}}{(m+1)!}.
\]
Since \(d>Kb^{m/(m+1)}\) and \(b\ge(1-\gamma)s\), the choice of \(K\) in
\eqref{eq:caseIII-K-choice} implies
\[
    |\Y_{m+1}|
    >
    \frac{K^{m+1}(1-\gamma)^m}{(m+1)!}s^m
    >
    Cs^m
    \ge
    |\F_{\le m}|.
\]
Hence
\(|\F_{\le m}| < |\Y_{m+1}| \le \Lambda+|\Y_{\ge m+1}|.\)
Thus \eqref{eq:case-target} holds strictly in this subcase.

This completes Case III.

\subsection{Completion of the proof and equality cases}

The three cases prove \eqref{eq:case-target} for every family with
\(\emptyset\notin\F\).  Hence \cref{lem:low-layer-comparison} gives
\(|\F|\le |P(m,s,\ell)|\).  It remains to discuss equality.

Suppose first that \(\emptyset\notin\F\) and
\(|\F|=|P(m,s,\ell)|\).  The good-set subcase of Case I and Cases II and
III are strict.  Hence equality can occur only in the no-good subcase of
Case I.  In that subcase, equality in \eqref{eq:case-target} forces $|\Y_{\ge m+1}|=0$ and $|\F_{\le m}|=\Lambda$.
Moreover equality must hold in Claim~\ref{cl:counting-from-missing}.  Thus
\(M=\binom Rm\), and
\[
    B_j(M)=
    \left\{E\in\binom{[n]}j: |E\cap A|\ge m+1-j\right\}
\]
for every \(0\le j\le m-1\).  Since \(\F_j\subseteq B_j(M)\) and the
total sizes are equal, we have \(\F_j=B_j(M)\) for all
\(0\le j\le m-1\).  Also
\[
    \F_m=\binom{[n]}m\setminus\binom Rm
    =
    \left\{E\in\binom{[n]}m:E\cap A\ne\emptyset\right\}.
\]
Together with \(\F_{\ge m+1}=\binom{[n]}{\ge m+1}\), this gives
\(\F=P(m,s,\ell;A).\)

Finally, if an extremal family originally contained \(\emptyset\), replacing
\(\emptyset\) by a missing singleton gives another extremal family with no
empty set but with a singleton.  This is impossible, since
\(P(m,s,\ell;A)\) contains no singleton for \(m\ge3\).  Hence no extremal
family contains \(\emptyset\), and the equality case follows.

\section{Proof of Theorem~\ref{thm:m3-lower-range}}\label{sec:m3-threshold}

We use the notation from \cref{sec:comparison} with $m=3$, but
we now work in the larger range $1\le \ell\le t(s)$.
The same comparison framework applies; the improvement uses sharper $3$-uniform matching results and better estimates.
Write $n=4s-\ell, a=\ell-1,r=n-a=4s-2\ell+1$, $\Lambda = \binom a2+\binom n3-\binom r3$ and $|P(3,s,\ell)|=\Lambda+\binom n{\ge4}$.

For a family \(\F\subseteq2^{[n]}\), write
\(H\coloneqq \F_3.\)
For \(E\in\F_j\), \(j\le3\), recall that \(\Delta(E)=4-j\), and for a family \(\mathcal Q\), \(\Delta(\mathcal Q)=\sum_{Q\in\mathcal Q}\Delta(Q)\).
We also keep the notation from \cref{sec:comparison}: for example,
\(H(\overline E)=\{Q\in H:Q\cap E=\emptyset\}\) is the family of triples in \(H\) disjoint from \(E\).
Finally, put
\(A_3\coloneqq \binom{3\ell-1}{3}.\)
The number $A_3$ is the size of the third layer of the second candidate
extremal family $P'(s,\ell;L')$.  Thus the comparison between $\Lambda$
and $A_3$ determines which of the two candidates can be extremal.

Recall that
\[
    t(s)
    =
    \frac{17-18s+\sqrt{49-852s+1284s^2}}{20}
    =\alpha_*s+O(1),
    \qquad
    \alpha_*=\frac{\sqrt{321}-9}{10}<1.
\]
The case \(\ell\le s/2\) follows from \cref{thm:main}, applied with
\(m=3\) and \(\eps=1/14\).
Hence, throughout the proof below, we assume
\begin{equation}\label{eq:m3-threshold-range}
    \frac s2<\ell\le t(s).
\end{equation}
Choose a fixed constant \(\lambda\) with \(\alpha_*<\lambda<1\).  Since
\(t(s)=\alpha_*s+O(1)\), for all sufficiently large \(s\) we have
\begin{equation}\label{eq:m3-lambda-range}
    \ell\le \lambda s.
\end{equation}
Since $s$ is sufficiently large, we have $\ell\ge6$.
Hence for \(j\in\{1,2\}\), the quantities \(p=\ell+j-4\) and \(p-1\) are positive.

The goal of this section is to prove \cref{thm:m3-lower-range}.  The proof
first compares the endpoint quantity \(A_3\) with \(\Lambda\), then reduces
the upper bound to the first three layers.  After that reduction, the argument
splits according to the matching number and vertex-cover number of
\(H=\F_3\): Case 1 treats the coverable case \(\tau(H)\le a\); Case 2
treats \(\nu(H)\le a<\tau(H)\) using the 3-uniform stability input; and
Case 3 treats \(\nu(H)\ge a+1\) by forcing enough missing 4-sets.

We first record a lemma concerning numerical estimates.
\begin{lemma}\label[lemma]{lem:m3-endpoint-comparison}
    We have $\Lambda\ge A_3$.
    Moreover, equality holds if and only if \(\ell=t(s)\).
\end{lemma}
\begin{proof}
    By the definitions of \(\Lambda\) and \(A_3\),
    \[
        |P(3,s,\ell)|=\Lambda+\binom n{\ge4}
        \quad\text{and}\quad
        |P'(s,\ell)|=A_3+\binom n{\ge4}.
    \]
    Hence
    \(\Lambda-A_3=|P(3,s,\ell)|-|P'(s,\ell)|\).
    A direct calculation gives
    \[
        \Lambda-A_3
        =
        \frac{(\ell-1)\bigl(-10\ell^2-18s\ell+17\ell+24s^2-6s-6\bigr)}{3}.
    \]
    The quadratic factor has roots
    \[
        \frac{17-18s-\sqrt{49-852s+1284s^2}}{20}
        \quad\text{and}\quad
        t(s)=\frac{17-18s+\sqrt{49-852s+1284s^2}}{20}.
    \]
    The first root is negative for all sufficiently large \(s\).
    Since \eqref{eq:m3-threshold-range} gives \(1<\ell\le t(s)\),
    the quadratic factor is nonnegative, with equality if and only if
    \(\ell=t(s)\).  As \(\ell-1>0\), the claim follows.
\end{proof}

We now reduce the proof to the first three layers.  Since all sets of size at
least $4$ are present in the canonical family, the only possible surplus of
\(\F\) again lies in the first three layers.
As in the proof of \cref{thm:main}, it suffices for the upper bound to
consider families with \(\emptyset\notin\F\).  We shall prove
\begin{equation}\label{eq:m3-key-low}
    |\F_1|+|\F_2|+|H|
    \le
    \Lambda+|\Y_{\ge4}|.
\end{equation}
Indeed, \eqref{eq:m3-key-low} gives
\[
    \begin{aligned}
        |\F|=
        |\F_1|+|\F_2|+|H|+\binom n{\ge4}-|\Y_{\ge4}| \le
        \Lambda+\binom n{\ge4}
        =
        |P(3,s,\ell)|.
    \end{aligned}
\]

We prove \eqref{eq:m3-key-low} by splitting according to the structure of  \(H\).
\[
    \begin{array}{ll}
        \text{Case 1:} & \tau(H)\le a,        \\[1mm]
        \text{Case 2:} & \nu(H)\le a<\tau(H), \\[1mm]
        \text{Case 3:} & \nu(H)\ge a+1=\ell .
    \end{array}
\]

\subsection{Proof of Case 1}

Let \(A\) be a vertex cover of \(H\) with \(|A|=a\).  Put
\(R=[n]\setminus A\) and \(M=\binom{[n]}3\setminus H\).
Since \(A\) covers \(H\), every triple contained in \(R\) is missing from
\(H\); that is,
\(\binom R3\subseteq M.\)
As in Case I of the general proof, we classify low-layer sets according to
whether they can be accounted for by missing triples.  The thresholds below are
the $m=3$ specialization of the general bad-set test, simplified for
singletons and pairs.  Recall that \(M(\overline E)\) denotes the missing triples disjoint from \(E\).
For \(j=1,2\), define
\[
    B_1(M)
    \coloneqq
    \left\{x\in[n]: |M(\overline{\{x\}})|\ge \binom{r+2}{3}\right\},
\]
and
\[
    B_2(M)
    \coloneqq
    \left\{E\in\binom{[n]}2: |M(\overline E)|\ge \binom r3\right\}.
\]
For \(j=1,2\) and \(E\in\F_j\), call \(E\) \emph{bad} if
\(E\in B_j(M)\), and \emph{good} otherwise.

We first record the counting estimate for $B_i(M)$.
Its proof is very similar to that of \cref{cl:counting-from-missing}, but requires a more accurate count.

\begin{claim}\label[claim]{cl:m3-terminal-core}
    We have
    \[
        |B_1(M)|+|B_2(M)|
        \le
        \binom a2+|M|-\binom r3.
    \]
    Moreover, equality implies $M=\binom R3$, $B_1(M)=\emptyset$ and $B_2(M)=\binom A2$.
\end{claim}

\begin{proof}
    Put
    \(M^+\coloneqq M\setminus\binom R3\) and \(\xi\coloneqq |M^+|=|M|-\binom r3\).
    Thus, for \(T\subseteq[n]\),
    \(M^+(\overline T)=\{Q\in M^+:Q\cap T=\emptyset\}.\)

    If \(E\in\binom A2\), then all triples in \(\binom R3\) are disjoint
    from \(E\).  Hence \(E\in B_2(M)\).  These sets contribute exactly
    \(\binom a2\) elements to \(B_2(M)\).

    Put
    \(\mathcal U \coloneqq \left(B_2(M)\setminus\binom A2\right)\cup B_1(M).\)
    We prove that \(|\mathcal U|\le\xi\).  Decompose \(\mathcal U\) into
    the four classes
    \[
        \begin{aligned}
             & \mathcal U_1\coloneqq B_2(M)\cap\{E\in\binom{[n]}2: |E\cap A|=1\}, \qquad
            \mathcal U_2\coloneqq B_2(M)\cap\binom R2,                                   \\
             & \mathcal U_3\coloneqq B_1(M)\cap A, \qquad  \mathcal U_4
            \coloneqq B_1(M)\cap R .
        \end{aligned}
    \]
    Clearly, $|\mathcal U_1|\le ar$, $|\mathcal U_2|\le\binom r2$, $|\mathcal U_3|\le a$ and $|\mathcal U_4|\le r$.
    Moreover, similar to the proof of \cref{cl:counting-from-missing}, we have
    \begin{itemize}
        \item if \(E\in\mathcal U_1\), then $\xi\ge |M^+(\overline E)|\ge\binom r3-\binom{r-1}3=\binom{r-1}{2}=D_1$;
        \item if \(E\in\mathcal U_2\), then $\xi\ge |M^+(\overline E)|\ge\binom r3-\binom{r-2}3=(r-2)^2=D_2$;
        \item if \(x\in\mathcal U_3\), then $\xi\ge |M^+(\overline{\{x\}})|
                  \ge \binom{r+2}3-\binom r3=r^2=D_3$;
        \item if \(x\in\mathcal U_4\), then $\xi\ge |M^+(\overline{\{x\}})|
                  \ge\binom{r+2}3-\binom{r-1}3=\frac{3r^2-3r+2}{2}=D_4$.
    \end{itemize}
    For \(i=1,2,3,4\), let $N_i$ be the total number of possible elements in \(\mathcal U_1\cup\cdots\cup\mathcal U_i\), that is,
    \[
        N_1=ar,
        \qquad
        N_2=ar+\binom r2,
        \qquad
        N_3=ar+\binom r2+a,
        \qquad
        N_4=ar+\binom r2+a+r.
    \]

    Since \(r-2a=4s-4\ell+3\) and \(\ell\le t(s)<s\), we have
    \(r\ge2a+7\) for all sufficiently large \(s\).
    We have
    \[
        D_2-D_1=\frac{(r-3)(r-2)}2,
        \qquad
        D_3-D_2=4(r-1),
        \qquad
        D_4-D_3=\frac{(r-2)(r-1)}2,
    \]
    which are positive for \(r\ge4\).
    Therefore, we obtain $D_1<D_2<D_3<D_4$.
    Also, for fixed \(a\), increasing
    \(r\) by one changes the four differences \(D_i-N_i\) by $r-a-1,r-a-3,r-a+1,2r-a-1$, respectively.
    These quantities are positive when \(r\ge2a+7\), so each
    \(D_i-N_i\) is minimized at \(r=2a+7\).  At that endpoint,
    \[
        D_1-N_1=4a+15,
        \qquad
        D_2-N_2=4,
        \qquad
        D_3-N_3=7a+28,
        \qquad
        D_4-N_4=2a^2+16a+36.
    \]
    Hence
    \(N_i\le D_i-1\) for \(i=1,2,3,4\).

    Now compare \(\xi\) with \(D_1<D_2<D_3<D_4\).  If \(\xi<D_1\), then
    all four families \(\mathcal U_i\) are empty.  If
    \(D_i\le\xi<D_{i+1}\) for some \(1\le i<4\), then only the first \(i\)
    families can be nonempty, and hence
    \(|\mathcal U|\le N_i\le D_i-1\le\xi.\)
    If \(\xi\ge D_4\), then
    \(|\mathcal U|\le N_4\le D_4-1\le\xi.\)
    Therefore \(|\mathcal U|\le\xi\), and so
    \[
        |B_1(M)|+|B_2(M)|
        \le
        \binom a2+\xi
        =
        \binom a2+|M|-\binom r3.
    \]

    The same threshold argument gives \(|\mathcal U|<\xi\) whenever
    \(\xi>0\).  Hence equality is possible only when \(\xi=0\), that is,
    when \(M=\binom R3\).  In this case the definitions give
    \(B_1(M)=\emptyset\) and \(B_2(M)=\binom A2\).
\end{proof}

We now prove \eqref{eq:m3-key-low} in Case 1.  First suppose that there are
no good sets.  Then
\(\F_1\subseteq B_1(M)\) and \(\F_2\subseteq B_2(M)\).
By \cref{cl:m3-terminal-core},
\(|\F_1|+|\F_2| \le \binom a2+|M|-\binom r3.\)
Adding \(|H|=\binom n3-|M|\) gives
\[
    |\F_1|+|\F_2|+|H|
    \le
    \binom a2+\binom n3-\binom r3
    =
    \Lambda.
\]
Thus \eqref{eq:m3-key-low} holds.

It remains to consider the case in which a good set exists.
Fix a good set \(E\in\F_j\), where \(j\in\{1,2\}\), and put
\(p\coloneqq \ell+j-4\) and \(N\coloneqq n-j\).
Since \(E\) is good,
\(|M(\overline E)|<\binom{r+4-2j}{3}.\)
Using \(r+4-2j=(n-j)-p+1\), and recalling that \(H(\overline E)\) denotes the triples in \(H\) disjoint from \(E\), we get
\begin{equation}\label{eq:m3-good-HbarE-lower}
    |H(\overline E)|
    >
    \binom N3-\binom{N-p+1}3 = \max\left\{\binom N3-\binom{N-p+1}3,\binom{{3p-1}}{3}\right\},
\end{equation}
where the last equality follows from \cref{lem:m3-good-emc-dominance}.
By Theorem~\ref{thm:m3-EMC3}, \(H(\overline E)\)
contains \(p\) pairwise disjoint triples.

Together with \(E\), these triples form a pairwise disjoint family
\(\mathcal Q\subseteq\F_1\cup\F_2\cup H\) whose deficit is
\(\Delta(\mathcal Q)=p+(4-j)=\ell.\)
By \cref{lem:deficit-completion} and \eqref{eq:m3-lambda-range},
\(|\Y_4|\ge c_0s^3\) for some constant \(c_0>0\).  Since
\(|\F_1|+|\F_2| \le n+\binom n2 =O(s^2),\)
we have \(|\F_1|+|\F_2|<|\Y_4|\) for all sufficiently large \(s\).  Also,
since \(A\) is a vertex cover of \(H\),
\(|H|\le \binom n3-\binom r3.\)
Therefore
\[
    |\F_1|+|\F_2|+|H|
    <
    \binom n3-\binom r3+|\Y_4|
    \le
    \Lambda+|\Y_{\ge4}|.
\]
Thus \eqref{eq:m3-key-low} holds strictly in this subcase.  This completes
Case 1.

\subsection{Proof of Case 2}

In this case, it follows from \(\ell\le t(s)<s\) that \(n-(3a+2) = 4s-\ell-(3\ell-1) = 4(s-\ell)+1>0\).
Hence \cref{thm:m3-GLM} gives
\[
    |H|
    \le
    \max\left\{
    h_3(n,a),\binom{3a+2}{3}
    \right\},
\]
where
\(h_3(n,a)= \binom n3-\binom{n-a}{3}+1-\binom{n-a-3}{2}.\)
Since \(3a+2=3\ell-1\), the second term is \(A_3\).

Since $\nu(\F)<s$, we have \(|\F_1|\le s-1\) and \(\nu(\F_2)<s\).
Since $n=4s-\ell>3s$, applying \cref{lem:EG} gives
\[
    |\F_2|
    \le
    \max\left\{\binom{2s-1}{2},\binom n2-\binom{n-s+1}{2}\right\}\le
    \binom n2-\binom{n-s+1}{2}
    =
    \frac{(s-1)(2n-s)}2,
\]
where the last inequality follows from \(n=4s-\ell>3s\).
Put
\(L_{12}\coloneqq (s-1)+\frac{(s-1)(2n-s)}2.\)
Thus
\begin{equation}\label{eq:m3-F12-L12}
    |\F_1|+|\F_2|\le L_{12}.
\end{equation}

We now split according to which term in the maximum controls \(|H|\).

\medskip
\noindent\textbf{Subcase a: \(|H|\le h_3(n,a)\).}
\medskip

By \cref{lem:m3-h3-gap},
\(\Lambda-h_3(n,a)>L_{12}.\)
Together with \eqref{eq:m3-F12-L12}, this gives
\(|\F_1|+|\F_2|+|H| < \Lambda.\)
Thus \eqref{eq:m3-key-low} holds strictly in Subcase a.

\medskip
\noindent\textbf{Subcase b: \(|H|\le A_3\).}
\medskip

Put
\(\delta\coloneqq A_3-|H|\ge0\) and \(g_0\coloneqq \Lambda-A_3\ge0\).
If \(\F_1\cup\F_2=\emptyset\), then $|\F_1|+|\F_2|+|H|=|H|\le A \le\Lambda$.
Thus \eqref{eq:m3-key-low} holds in this subcase.
Equality here can occur
only when \(\ell=t(s)\) and \(|H|=A_3\).
Hence, for the rest of Subcase b,
we assume that \(E\in\F_j\) for some \(j\in\{1,2\}\).

Write $p\coloneqq \ell+j-4$ and $N\coloneqq n-j$.
For \(i\in\{1,2\}\), write $p_i\coloneqq \ell+i-4$, $N_i\coloneqq n-i$ and $L_i\coloneqq \binom n3-\binom{n-i}{3}$.
Thus \(p=p_j\) and \(N=N_j\), and \(L_j\) is the number of triples meeting a fixed \(j\)-set.

\begin{lemma}\label[lemma]{lem:m3-case2-numerical}
    For all sufficiently large \(s\) and every \(i\in\{1,2\}\), we have
    \[
        \Lambda-L_i-h_3(N_i,p_i-1)-L_{12}>0 \qquad\text{and}\qquad \Lambda-L_i-\binom{3p_i-1}{3}-L_{12}>0.
    \]
\end{lemma}

The proof of \cref{lem:m3-case2-numerical} contains some tricky computations, and hence it is deferred in \cref{app:m3-numerics}.
The next claim handles the situation where a chosen low-layer set cannot be
completed by disjoint triples.  Here \(H(\overline E)\) again denotes the triples of \(H\) disjoint from the fixed set \(E\).
In that case the obstruction implies $H$ does not contain too many members.

\begin{claim}\label[claim]{cl:m3-subcaseb-no-p-matching}
    If \(H(\overline E)\) contains no \(p\) pairwise disjoint triples, then
    \begin{equation}\label{eq:m3-A3-delta-large}
        \delta>L_{12}-g_0.
    \end{equation}
\end{claim}

\begin{proof}
    Assume that \(\nu(H(\overline E))\le p-1\).  We first show that
    \(\tau(H(\overline E))>p-1\).  Indeed, if
    \(\tau(H(\overline E))\le p-1\), then a vertex cover of
    \(H(\overline E)\), together with the vertices of \(E\), covers every
    triple of \(H\).  The size of this cover is at most
    \((p-1)+j = \ell+2j-5 \le \ell-1 = a,\)
    contradicting \(\tau(H)>a\).  Thus
    \(\nu(H(\overline E))\le p-1<\tau(H(\overline E)).\)
    Also, by the choice of the absolute threshold, \(N\ge N_0\), where
    \(N_0\) is the threshold in Theorem~\ref{thm:m3-GLM}.  Moreover,
    \(N-(3(p-1)+2) = N-(3p-1) = 4(s-\ell)+13-4j>0\)
    for all sufficiently large \(s\).  Hence Theorem~\ref{thm:m3-GLM} gives
    \[
        |H(\overline E)|
        \le
        \max\left\{
        h_3(N,p-1),\binom{3p-1}{3}
        \right\}.
    \]
    Since at most \(L_j\) triples meet \(E\), $|H|\le L_j+\max\left\{h_3(N,p-1),\binom{3p-1}{3}\right\}$.
    Equivalently,
    \[
        \delta
        \ge
        A_3-L_j-
        \max\left\{
        h_3(N,p-1),\binom{3p-1}{3}
        \right\}.
    \]
    Since \(g_0=\Lambda-A_3\), \cref{lem:m3-case2-numerical} with
    \(i=j\) gives
    \(A_3-L_j-h_3(N,p-1)>L_{12}-g_0\)
    and
    \(A_3-L_j-\binom{3p-1}{3}>L_{12}-g_0.\)
    Hence \eqref{eq:m3-A3-delta-large} follows.
\end{proof}

If \(H(\overline E)\) contains no \(p\) pairwise disjoint triples, then
\cref{cl:m3-subcaseb-no-p-matching} gives \(\delta>L_{12}-g_0\).  Hence
\[
    \Lambda-|H|
    =
    g_0+\delta
    >
    L_{12}
    \ge
    |\F_1|+|\F_2|.
\]
Thus
\(|\F_1|+|\F_2|+|H|<\Lambda,\)
and \eqref{eq:m3-key-low} holds strictly.

Otherwise, \(H(\overline E)\) contains \(p\) pairwise disjoint triples.
Together with \(E\), these triples form a pairwise disjoint family
\(\mathcal Q\subseteq\F_1\cup\F_2\cup H\) with
\(\Delta(\mathcal Q)=p+(4-j)=\ell.\)
By \cref{lem:deficit-completion} and \eqref{eq:m3-lambda-range},
\(|\Y_4|\ge c_0s^3\) for some constant \(c_0>0\).  Since
\(|\F_1|+|\F_2|\le L_{12}=O(s^2)\), we have
\(|\F_1|+|\F_2|<|\Y_4|\)
for all sufficiently large \(s\).  Since \(|H|\le A_3\le\Lambda\), we obtain $|\F_1|+|\F_2|+|H| < \Lambda+|\Y_4| \le \Lambda+|\Y_{\ge4}|$.
Thus \eqref{eq:m3-key-low} holds strictly in Subcase b.
This completes Case 2.

\subsection{Proof of Case 3}

Write
\(q\coloneqq \nu(H)=\ell+d\)
for some \(d\ge0\).  Since \(\nu(\F)<s\), we have \(q\le s-1\).

We first bound the possible excess of the first three layers over
\(\Lambda\).  Unlike in the general proof, Frankl's exact three-uniform
bound has two possible extremal terms, so both terms must be compared with
the canonical third layer.  Since \(q+1\le s\) and \(s\le(n+1)/3\),
Theorem~\ref{thm:m3-EMC3} gives
\[
    |H|
    \le
    \max\left\{
    \binom n3-\binom{n-q}{3},
    \binom{3q+2}{3}
    \right\}.
\]
For the first term,
\(n-q=4s-\ell-(\ell+d)=r-1-d,\)
and hence
\[
    \binom n3-\binom{n-q}{3}
    -
    \left(\binom n3-\binom r3\right)
    =
    \binom r3-\binom{r-1-d}{3}
    \le
    (d+1)\binom r2
    =
    O((d+1)s^2).
\]
For the second term,
\(\binom{3q+2}{3}-\binom{3\ell+2}{3}=O(ds^2),\)
and there is a constant \(C_0\) such that
\(\binom{3\ell+2}{3} - \left(\binom n3-\binom r3\right) \le C_0s^2.\)
Indeed, \(\binom{3\ell+2}{3}-A_3=O(s^2)\), and
\(A_3-\left(\binom n3-\binom r3\right) \le \binom a2=O(s^2)\)
because \(A_3\le\Lambda=\binom a2+\binom n3-\binom r3\).  Since
\(|\F_1|+|\F_2|\le n+\binom n2=O(s^2),\)
there is an absolute constant \(C\) such that
\begin{equation}\label{eq:m3-overflow-excess}
    |\F_1|+|\F_2|+|H|
    \le
    \Lambda+C(d+1)s^2.
\end{equation}

Choose \(q\) pairwise disjoint triples \(Q_1,\ldots,Q_q\in H\).
Put $t_0\coloneqq s-q$ and $W\coloneqq[n]\setminus\bigcup_{i=1}^q Q_i$.
Then \(|W|=n-3q=4t_0+d\).  If \(\F_4[W]\) contained \(t_0\) pairwise
disjoint 4-sets, then those 4-sets together with
\(Q_1,\ldots,Q_q\) would form an \(s\)-matching in \(\F\).  Hence
\(\nu(\F_4[W])<t_0\).  By \cref{lem:blocker}, with
\[
    T_1\coloneqq \frac1{t_0}\binom{4t_0+d}{4},
    \qquad
    T_2\coloneqq \binom{d+4}{4},
\]
we have
\begin{equation}\label{eq:m3-blocker-terms}
    |\Y_4|
    \ge
    \max\{T_1,T_2\}.
\end{equation}

Let
\(b\coloneqq s-\ell.\)
Then \(t_0+d=b\).  Since \(\ell\le t(s)\) and \(t(s)/s\to\alpha_*<1\),
there is a constant \(\beta>0\) such that
\(b\ge \beta s.\)
It remains to show that \(|\Y_4|\) is larger than the error term in
\eqref{eq:m3-overflow-excess}.  The first blocker term is used when
\(d\) is small, and the second blocker term is used when
\(d\) is large.  Choose a sufficiently large constant \(K\).

\medskip
\noindent\textbf{Subcase a: \(d\le Kb^{2/3}\).}
\medskip

Then \(t_0=b-d\ge b/2\) for all sufficiently large \(s\).  Hence
\(T_1 = \frac1{t_0}\binom{4t_0+d}{4} = \Omega(b^3) = \Omega(s^3).\)
On the other hand,
\(C(d+1)s^2=O(b^{8/3}).\)
Thus
\(|\Y_4|>C(d+1)s^2\)
for all sufficiently large \(s\).

\medskip
\noindent\textbf{Subcase b: \(d>Kb^{2/3}\).}
\medskip

The second term in \eqref{eq:m3-blocker-terms} gives
\(T_2 = \binom{d+4}{4} = \Omega(d^4).\)
Since \(s=O(b)\),
\(C(d+1)s^2=O(db^2).\)
Choosing \(K\) large enough gives
\(|\Y_4|>C(d+1)s^2.\)

In both subcases, \(|\Y_4|>C(d+1)s^2\).  Combining this with
\eqref{eq:m3-overflow-excess}, we obtain
\[
    |\F_1|+|\F_2|+|H|
    <
    \Lambda+|\Y_4|
    \le
    \Lambda+|\Y_{\ge4}|.
\]
Hence \eqref{eq:m3-key-low} holds strictly in Case 3.

\subsection{Completion of the proof and equality cases}

The three cases above prove \eqref{eq:m3-key-low} for every family with
\(\emptyset\notin\F\) satisfying \eqref{eq:m3-threshold-range}.  Hence
\(|\F|\le |P(3,s,\ell)|.\)
The canonical family \(P(3,s,\ell;A)\) has no \(s\) pairwise disjoint members
and has size \(|P(3,s,\ell)|\).  If \(t(s)\) is an integer and \(\ell=t(s)\), then
\[
    |P(3,s,\ell;A)|
    =
    \Lambda+\binom n{\ge4}
    =
    A_3+\binom n{\ge4}
    =
    |P'(s,\ell;A')|
\]
for every \(A'\in\binom{[n]}{3\ell-1}\).  Moreover,
\(P'(s,\ell;A')\) has no \(s\) pairwise disjoint members.  Indeed, if an
\(s\)-matching in \(P'(s,\ell;A')\) contains \(q\) triples, then its total
number of vertices is at least
\(3q+4(s-q)=4s-q.\)
Since \(n=4s-\ell\), this forces \(q\ge\ell\).  But all triples of
\(P'(s,\ell;A')\) lie inside \(A'\), and \(|A'|=3\ell-1\), so \(A'\) cannot
contain \(\ell\) pairwise disjoint triples.

It remains to classify equality.  The strict inequalities above show that
every subcase is strict except for two possibilities: the subcase with no good sets in
the coverable case, and the branch of Case 2 using the complete $3$-graph term with no lower-layer
sets.  Suppose first that \(\emptyset\notin\F\) and
\(|\F|=|P(3,s,\ell)|\).  Case 3 and the good-set subcase of Case 1 are
strict.  In Case 2, Subcase a is strict, and Subcase b is strict whenever
\(\F_1\cup\F_2\ne\emptyset\).  Hence equality can occur only in the
following two situations.

First, equality may occur in the subcase of Case 1 with no good sets.  Then equality in
\eqref{eq:m3-key-low} forces
\[
    |\Y_{\ge4}|=0
    \qquad\text{and}\qquad
    |\F_1|+|\F_2|+|H|=\Lambda.
\]
Therefore equality holds in \cref{cl:m3-terminal-core}.  Hence, we obtain $\F_1=\emptyset$, $\F_2=\binom A2$ and
\[
    H=
    \left\{T\in\binom{[n]}3:T\cap A\ne\emptyset\right\}
\]
for some \(A\subseteq[n]\) with \(|A|=a=\ell-1\) and $\F_{\ge4}=\binom{[n]}{\ge4}$.
Thus
\(\F=P(3,s,\ell;A).\)

Second, equality may occur in Case 2, Subcase b, with
\(\F_1\cup\F_2=\emptyset\).  In this case equality forces
\(\ell=t(s)\), \(|H|=A_3=\binom{3\ell-1}{3}\), and \(|\Y_{\ge4}|=0\).
Since Case 2 gives \(\nu(H)\le a=\ell-1\), we have \(\nu(H)<\ell\).  At the
endpoint,
\(A_3- \left(\binom n3-\binom r3\right) = \binom a2>0,\)
because \(\Lambda=A_3\) and
\(\Lambda=\binom a2+\binom n3-\binom r3.\)
Moreover,
\(n-(3\ell-1)=4(s-\ell)+1>0,\)
so Theorem~\ref{thm:m3-EMC3} applies with matching parameter \(\ell\).
By the equality case in Theorem~\ref{thm:m3-EMC3}, it follows that
\(H=\binom{A'}3\)
for some \(A'\in\binom{[n]}{3\ell-1}\).  Therefore
\(\F=\binom{A'}3\cup\binom{[n]}{\ge4} = P'(s,\ell;A').\)

Thus, if \(\ell<t(s)\), equality is possible only for the canonical family
\(P(3,s,\ell;A)\).  If \(t(s)\) is an integer and \(\ell=t(s)\), equality is possible only for
\(P(3,s,\ell;A)\) or \(P'(s,\ell;A')\).
Finally, if an extremal family originally contained \(\emptyset\), replacing
\(\emptyset\) by a missing singleton gives another extremal family with no
empty set but with a singleton.  This is impossible, since neither extremal
family described above contains a singleton.  Hence no extremal family
contains \(\emptyset\), completing the proof of \cref{thm:m3-lower-range}.

\appendix

\section{Numerical estimates in the proof of Theorem~\ref{thm:m3-lower-range}}
\label{app:m3-numerics}

\begin{lemma}\label[lemma]{lem:m3-good-emc-dominance}
    Assume \eqref{eq:m3-threshold-range}.  Let \(j\in\{1,2\}\), put
    \[
        p=\ell+j-4,
        \qquad
        N=n-j=4s-p-4.
    \]
    Then, for all sufficiently large \(s\),
    \(N\ge3p-1\)
    and
    \[
        \binom N3-\binom{N-p+1}3>\binom{3p-1}3.
    \]
\end{lemma}

\begin{proof}
    We have
    \(N-3p+1=4(s-\ell)+13-4j>0\)
    for all sufficiently large \(s\), since \(\ell\le t(s)<s\).
    Hence \(N\ge3p-1\).

    Put
    \[
        R(p,s)
        \coloneqq
        48s^2-36sp-20p^2+67p-108s+54.
    \]
    Direct calculation gives
    \[
        \binom N3-\binom{N-p+1}3-\binom{3p-1}3
        =
        \frac{p-1}{6}R(p,s).
    \]
    Moreover,
    \(\frac{\partial R}{\partial p}=-36s-40p+67<0\)
    on the range \(s/2-3<p=\ell+j-4\le t(s)-2\).  Hence
    \(R(p,s)\ge R(t(s)-2,s).\)
    Using
    \(10t(s)^2+(18s-17)t(s)-24s^2+6s+6=0,\)
    we get
    \[
        \begin{aligned}
            R(t(s)-2,s)
             & =
            48s^2-36s(t(s)-2)-20(t(s)-2)^2+67(t(s)-2)-108s+54 \\
             & =113t(s)-24s-148                               \\
             & =(113\alpha_*-24)s+O(1)>0.
        \end{aligned}
    \]
    Since \(p-1>0\) for all sufficiently large \(s\), the desired inequality
    follows.
\end{proof}

\begin{lemma}\label[lemma]{lem:m3-h3-gap}
    Under \eqref{eq:m3-threshold-range}, for all sufficiently large \(s\),
    \(\Lambda-h_3(n,a)>L_{12},\)
    where \(L_{12}\) is defined in \eqref{eq:m3-F12-L12}.
\end{lemma}

\begin{proof}
    Direct calculation gives
    \[
        2\bigl(\Lambda-h_3(n,a)-L_{12}\bigr)
        =
        5\ell^2-14\ell s+5\ell+9s^2-15s+8.
    \]
    The quadratic part is
    \[
        5\ell^2-14\ell s+9s^2
        =
        (s-\ell)(9s-5\ell).
    \]
    Since \(t(s)/s\to\alpha_*<1\), both \(s-\ell\) and \(9s-5\ell\)
    are bounded below by positive constant multiples of \(s\) throughout
    the range \(s/2<\ell\le t(s)\).  Hence the quadratic part is
    \(\Omega(s^2)\), and the lower-order terms are negligible.  Therefore
    \(\Lambda-h_3(n,a)>L_{12}\) for all sufficiently large \(s\).
\end{proof}

\begin{proof}[Proof of \cref{lem:m3-case2-numerical}]
    For the term \(h_3(N_i,p_i-1)\), direct expansion gives
    \[
        \Lambda-L_1-h_3(n-1,\ell-4)-L_{12}
        =
        \frac{
            13\ell^2-46\ell s-11\ell+41s^2+17s+4
        }2,
    \]
    and
    \[
        \Lambda-L_2-h_3(n-2,\ell-3)-L_{12}
        =
        \frac{
            5\ell^2-14\ell s+5\ell+9s^2-15s+8
        }2.
    \]
    Their quadratic parts are respectively
    \[
        \frac{13x^2-46x+41}{2}s^2
        \quad\text{and}\quad
        \frac{(x-1)(5x-9)}2s^2,
        \qquad x=\ell/s.
    \]
    Since \(t(s)/s\to\alpha_*<1\), the relevant range satisfies
    \(1/2<x\le\alpha_*+o(1)\).  On this range both quadratic coefficients
    are bounded away from zero.  Hence, for all sufficiently large \(s\),
    \[
        \Lambda-L_1-h_3(n-1,\ell-4)-L_{12}>0,
        \qquad
        \Lambda-L_2-h_3(n-2,\ell-3)-L_{12}>0.
    \]

    For the term \(\binom{3p_i-1}{3}\), direct expansion gives
    \[
        \Lambda-L_1-\binom{3\ell-10}{3}-L_{12}
        =
        \frac{
            -20\ell^3-36\ell^2s+294\ell^2+48\ell s^2
            +54\ell s-1114\ell-117s^2+63s+1326
        }6,
    \]
    and
    \[
        \Lambda-L_2-\binom{3\ell-7}{3}-L_{12}
        =
        \frac{
            -20\ell^3-36\ell^2s+210\ell^2+48\ell s^2
            +78\ell s-616\ell-165s^2+123s+492
        }6.
    \]
    As functions of \(\ell\), both have derivative with leading term
    \((-60x^2-72x+48)s^2\), where \(x=\ell/s\).  This leading coefficient is
    negative for \(x>1/2\).  Thus both expressions are decreasing in \(\ell\)
    throughout \(s/2\le\ell\le t(s)\).  It is enough to evaluate them at
    \(\ell=t(s)\).  Using the defining relation for \(t(s)\), the two
    expressions at \(\ell=t(s)\) have leading terms
    \[
        \left(\frac{724}{5}-\frac{67\sqrt{321}}{10}\right)s^2
        \quad\text{and}\quad
        \left(\frac{1923}{25}-\frac{189\sqrt{321}}{50}\right)s^2,
    \]
    respectively.  Both coefficients are positive.  Hence, for all
    sufficiently large \(s\),
    \[
        \Lambda-L_1-\binom{3\ell-10}{3}-L_{12}>0,
        \qquad
        \Lambda-L_2-\binom{3\ell-7}{3}-L_{12}>0.
    \]
\end{proof}


\begin{thebibliography}{99}

    \bibitem{Baranyai}
    Zs. Baranyai,
    On the factorization of the complete uniform hypergraph,
    in \emph{Infinite and Finite Sets, Vol.~I},
    Colloq. Math. Soc. J\'anos Bolyai, Vol.~10,
    North-Holland, Amsterdam, 1975, pp.~91--108.

    \bibitem{ErdosEMC}
    P. Erd\H{o}s,
    A problem on independent $r$-tuples,
    \emph{Ann. Univ. Sci. Budapest. E\"otv\"os Sect. Math.} {\bf 8} (1965), 93--95.

    \bibitem{ErdosGallai}
    P. Erd\H{o}s and T. Gallai,
    On maximal paths and circuits of graphs,
    \emph{Acta Math. Acad. Sci. Hungar.} {\bf 10} (1959), 337--356.

    \bibitem{ErdosKoRado}
    P. Erd\H{o}s, C. Ko and R. Rado,
    Intersection theorems for systems of finite sets,
    \emph{Quart. J. Math. Oxford Ser. (2)} {\bf 12} (1961), 313--320.

    \bibitem{FranklDAM}
    P. Frankl,
    On the maximum number of edges in a hypergraph with given matching number,
    \emph{Discrete Appl. Math.} {\bf 216} (2017), 562--581.

    \bibitem{FKnonuniform}
    P. Frankl and A. Kupavskii,
    Families with no $s$ pairwise disjoint sets,
    \emph{J. Lond. Math. Soc. (2)} {\bf 95} (2017), no.~3, 875--894.

    \bibitem{FKstability}
    P. Frankl and A. Kupavskii,
    Two problems on matchings in set families -- In the footsteps of
    Erd\H{o}s and Kleitman,
    \emph{J. Combin. Theory Ser. B} {\bf 138} (2019), 286--313.

    \bibitem{FK2022}
    P. Frankl and A. Kupavskii,
    The Erd\H{o}s Matching Conjecture and concentration inequalities,
    \emph{J. Combin. Theory Ser. B} {\bf 157} (2022), 366--400.

    \bibitem{GuoLuMao}
    M. Guo, H. Lu and D. Mao,
    A stability result on matchings in $3$-uniform hypergraphs,
    \emph{SIAM J. Discrete Math.} {\bf 36} (2022), no.~3, 2339--2351.

    \bibitem{Kleitman}
    D. J. Kleitman,
    Maximal number of subsets of a finite set no $k$ of which are pairwise disjoint,
    \emph{J. Combin. Theory} {\bf 5} (1968), 157--163.

    \bibitem{KupavskiiSokolov2025}
    A. Kupavskii and G. Sokolov,
    A complete solution of the Erd\H{o}s--Kleitman matching problem for $n\le3s$,
    arXiv:2511.21628, 2025.

    \bibitem{KupavskiiSokolovOtherEnd}
    A. Kupavskii and G. Sokolov,
    Families without $s$-matchings: the other end,
    arXiv:2605.00996, 2026.

    \bibitem{KupavskiiSokolovMore}
    A. Kupavskii and G. Sokolov,
    More on the Erd\H{o}s--Kleitman problem on matchings in set families,
    arXiv:2605.04379, 2026.

    \bibitem{LuczakMieczkowska}
    T. \L{}uczak and K. Mieczkowska,
    On Erd\H{o}s' extremal problem on matchings in hypergraphs,
    \emph{J. Combin. Theory Ser. A} {\bf 124} (2014), 178--194.

\end{thebibliography}
\end{document}